\documentclass[10pt]{amsart}

\usepackage[latin1]{inputenc}
\usepackage{amsfonts}
\usepackage{amssymb}
\usepackage{amsmath}
\usepackage{amsthm}
\usepackage{latexsym}
\usepackage{color}
\usepackage{amscd}
\usepackage[all,cmtip]{xy}
\usepackage{graphicx}
\usepackage{enumitem} 
\usepackage[margin=1.01in]{geometry}
\usepackage{esint} 
\usepackage{epsfig}
\usepackage[usenames,dvipsnames,svgnames,table]{xcolor}
\usepackage[colorlinks=true,pdfstartview=FitV,
linkcolor=ForestGreen,citecolor=ForestGreen,urlcolor=blue]{hyperref}
\usepackage{hyperref}
\usepackage{verbatim}

\numberwithin{equation}{section}

\newtheorem{theorem}{Theorem}[section]
\newtheorem{proposition}[theorem]{Proposition}
\newtheorem{lemma}[theorem]{Lemma}
\newtheorem{corollary}[theorem]{Corollary}

\theoremstyle{definition}
\newtheorem{definition}[theorem]{Definition}

\theoremstyle{remark}
\newtheorem{remark}[theorem]{Remark}
\newtheorem{remarks}[theorem]{Remarks}

\newcommand{\N}{{\mathbb N}}
\newcommand{\R}{{\mathbb R}}
\newcommand{\T}{{\mathbb T}}

\newcommand{\cC}{\mathcal C}

\newcommand{\cQ}{\mathcal Q}
\newcommand{\Qext}{\cQ_{\mathrm{ext}}}
\newcommand{\Qint}{\cQ_{\mathrm{int}}}
\newcommand{\cG}{\mathcal G}

\def\fR{\mathcal{R}}

\def\eps{\varepsilon}
\def\un{\mathbf{1}}

\newcommand\dt{{\frac{\mathrm d}{\mathrm dt}}}

\newcommand{\dd}{{\, \mathrm d}}

\DeclareMathOperator{\degkin}{deg_{\, kin}}
\DeclareMathOperator{\starkin}{\star_{\, kin}}

\begin{document}

\title[The Schauder estimate in kinetic theory and application]{The
  Schauder estimate in kinetic theory\\ with application to a toy
  nonlinear model}

\author[C. Imbert]{Cyril Imbert} \address[C.I.]{CNRS \& DMA, \'Ecole
  Normale Sup\'erieure,   45, Rue d'Ulm, 75005 Paris, France} \email{cyril.imbert@ens.fr}

\author[C. Mouhot]{Cl\'ement Mouhot}
\address[C.M.]{DPMMS, University
  of Cambridge, Wilberforce road, Cambridge CB3 0WA, UK}
\email{c.mouhot@dpmms.cam.ac.uk}

\thanks{The authors would like to thank L. Silvestre for fruitful
  discussions and detailed comments on the previous version of this
  work. This lead in particular to simplifying our use of two
  definitions of H\"older norms to only one.}

\date{\today}

\begin{abstract}
  This article is concerned with the Schauder estimate for linear
  kinetic Fokker-Planck equations with H\"older continuous
  coefficients. This equation has an hypoelliptic structure.  As an
  application of this Schauder estimate, we prove the global
  well-posedness of a toy nonlinear model in kinetic theory. This
  nonlinear model consists in a non-linear kinetic Fokker-Planck
  equation whose steady states are Maxwellian and whose diffusion in
  the velocity variable is proportional to the mass of the solution.
\end{abstract}

\maketitle

\tableofcontents

\section{Introduction}

\subsection{The Schauder estimate for linear kinetic Fokker-Planck
  equations}

The first part of this paper deals with Schauder estimate for linear
kinetic Fokker-Planck equations of the form
\begin{equation}
  \label{eq:linear}
  (\partial_t  + v \cdot \nabla_x) g =
  \sum_{1\le i,j\le d} a^{i,j} \partial^2_{v_iv_j} g
  + \sum_{1 \le i \le d}b^i \partial_{v_i} g + c g + S
  \quad \textit{ in } \R \times \R^d \times \R^d
\end{equation}
for some given function $S$ under the assumption that the diffusion
matrix $A= (a^{i,j}(t,x,v))_{i,j=1,\dots,d}$ satisfy a \emph{uniform
  ellipticity condition} for some $\lambda>0$:
\begin{equation}\label{eq:elliptic}
  \forall \, (t,x,v) \in \R \times \R^d \times \R^d, \
  \forall \, \xi \in \R^d, \quad
  \sum_{1\le i,j\le d} a^{i,j}(t,x,v) \xi_i \xi_j \ge \lambda |\xi|^2.
\end{equation}
The main result of this article is a \emph{Schauder estimate}, that is
to say an a priori estimate for classical solutions to
\eqref{eq:linear} controlling their second-order H\"older regularity
(in the sense of a ``kinetic order'' made precise below) by their
supremum norm, under the assumptions that the coefficients
$a^{i,j},b^i,c$ are also H\"older continuous. 
\begin{theorem}[The Schauder estimate]\label{thm:variable}
  Given $\alpha \in (0,1)$ and
  $a^{i,j},b^i,c \in C^\alpha(\R \times \R^d \times \R^d)$,
  $i,j=1,\dots,d$, satisfying \eqref{eq:elliptic} and a function
  $S \in C^\alpha(\R \times \R^d \times \R^d)$, any classical solution
  $g$ to \eqref{eq:linear} satisfies
  \begin{align}
    \label{eq:schauder-main}
    \left[ \partial_t g + v \cdot \nabla_x g
    \right]_{\cC^\alpha(\R \times \R^d \times \R^d)} + \left[ \nabla_v ^2 g
    \right]_{\cC^\alpha(\R \times \R^d \times \R^d)} \le C \left([S
      ]_{\cC^\alpha(\R \times \R^d \times \R^d)} +
      \|g\|_{L^\infty(\R \times \R^d \times \R^d)}\right)
  \end{align}
  where the constant $C$ depends on dimension $d$, the constant
  $\lambda$ from \eqref{eq:elliptic}, the exponent $\alpha$ and the
  $[\cdot]_{\cC^\alpha}$ semi-norm and $L^\infty$ norm of $a^{i,j}$ for
    $i,j=1,\dots,d$, $b^i$ for $i=1,\dots,d$, and $c$. The semi-norm
    $[\cdot]_{\cC^\alpha}$ is the standard H\"older semi-norm for the
    distance $\|(t,x,v)\|=|t|^{\frac12} + |x|^{\frac13} +|v|$.
\end{theorem}

\begin{remark}
  The left hand side can be understood as a H\"older regularity of
  (kinetic) order $2+\alpha$, according the specific definition of
  H\"older spaces $\cC_\ell^\beta$, $\beta \ge 0$, given in the next
  section (Definition~\ref{def:holder}). We compare this result to
  the classical Schauder estimate for parabolic equations in the next
  subsection.
\end{remark}
Such a Schauder estimate is typically used to reach well-posedness of
nonlinear equations after the derivation of H\"older estimates on
coefficients.  In order to illustrate this fact, we consider in the
second half of this paper the equation
\begin{equation}\label{eq:main}
  (\partial_t  + v \cdot \nabla_x) f
  = \rho[f] \nabla_v \cdot \left( \nabla_v f + v f \right)
\end{equation}
for an unknown $0 \le f =f(t,x,v)$, supplemented with the initial
condition $f (0,x,v)=f_{in}(x,v)$ in $\T^d \times \R^d$, where
$\rho[f] (t,x):= \int_v f(t,x,v) \dd v$ and $\T^d$ denotes the $d$-%
dimensional torus. We emphasize the fact that studying \eqref{eq:main}
with $x \in \T^d$ is equivalent to study it with $x \in \R^d$ with
periodic initial data. The known a priori estimates that are preserved
in time for this equation are $L^1(\T^d \times \R^d)$ and
$C_1 \mu \le f \le C_2 \mu$, where $\mu$ denotes the Gaussian
$(2\pi)^{-d/2} e^{-|v|^2/2}$. They are not sufficient to derive
uniqueness or boostrap higher regularity. The Schauder estimate from
Theorem~\ref{thm:variable}, together with the H\"older regularity
from~\cite{gimv} (see Theorem~\ref{t:gimv}), allows us to prove global
well-posedness of Eq.~\eqref{eq:main} in Sobolev spaces. In the
following statement $H^k(\T^d \times \R^d)$ denotes the standard
$L^2$-based Sobolev space.
\begin{theorem}[Global well-posedness for a toy nonlinear
  model]\label{theo:main-holder}
  Given two constants $0<C_1\le C_2$, let $f_{in}$ be such that
  $f_{in}/\sqrt{\mu} \in H^k(\T^d \times \R^{d})$ with $k > 2+d/2$ and
  satisfying $ C_1 \mu \le f_{in} \le C_2 \mu$. There then exists a
  unique global-in-time solution $f$ of \eqref{eq:main} in
  $(0,+\infty) \times \T^d \times \R^d$ satisfying
  $f(0,x,v) = f_{in} (x,v)$ everywhere in $\T^d \times \R^{d}$ and
  $f(t)/\sqrt{\mu} \in H^k(\T^d \times \R^{d})$ for all time $t>0$ and
  $C_1 \mu \le f \le C_2 \mu$ in
  $[0,+\infty) \times \T^d \times \R^d$.
\end{theorem}

\subsection{Schauder estimates for kinetic equations}

The Schauder estimate for solutions $g(t,v)$ to parabolic equations of
the form
\begin{eqnarray}
  \label{eq:parabolic}
  &  \displaystyle \partial_t g =
    \sum_{1\le i,j\le d} a^{i,j} \partial^2_{v_iv_j} g +
    \sum_{1 \le i \le d}b^i \partial_{v_i} g + c g + S \quad \textit{ in
    } \R \times \R^d \\[1ex]
  \nonumber
  \text{ takes the form } \qquad
  & [\partial_t g ]_{\cC^{\alpha}(\R \times \R^d)}
    + [D^2_v g]_{\cC^\alpha(\R \times \R^d)}  \le C
    \left( [S ]_{\cC^\alpha(\R \times  \R^d)} +    \|g\|_{L^\infty(\R
    \times  \R^d)} \right)
\end{eqnarray}
where $C^\alpha (\R \times \R^d)$ denotes the classical H\"older
space with respect to the parabolic distance
$\|(t,v)\|=|t|^{\frac12} + |v|$. This distance accounts for the
parabolic scaling $(t,v) \mapsto (r^2t,rv)$.

Because we work with kinetic Fokker-Planck equations, the usual
parabolic scaling is replaced with the kinetic scaling
$(t,x,v) \mapsto (r^2t,r^3x,rv)$. Moreover, parabolic equations of the
form \eqref{eq:parabolic} are translation invariant while kinetic
Fokker-Planck equations of the form \eqref{eq:linear} are translation
invariant in the space variable $x$ but not in the velocity one $v$;
the latter is replaced by the \emph{Galilean invariance}. As already
noticed for instance in \cite{polidoro1994,gimv}, these two facts
---kinetic scaling and Galilean invariance--- naturally require new
definitions for cylinders, order of polynomials and H\"older
continuity. We therefore define (see Definition~\ref{def:holder}) the
space $\cC^\beta_\ell$ to be the set of functions whose difference
with any polynomial of (kinetic) degree smaller than $\beta$ decays at
rate $r^\beta$ in a cylinder of radius $r>0$.
Following~\cite{imbert2018schauder}, one can define the \emph{kinetic
  degree} that follows the kinetic scaling, as
$2 (\mbox{degree in } t) + 3(\mbox{degree in } x) + (\mbox{degree in }
v)$. The subscript $\ell$ stands for ``left'' since the
transformations leaving the equation invariant are applied to the
left, see the next section.

There exists a well-developed literature of Schauder estimates for
ultraparabolic equations, e.g. \cite{polidoro1994,manfredini}, and
so-called Kolmogorov or H\"ormander type equations, see
e.g. \cite{dfp,bb,lunardi1997,radkevich2008} and references there in.
Some of these large classes of equations include the linear kinetic
Fokker-Planck equations of the form \eqref{eq:linear}. Moreover,
\eqref{eq:linear} is already considered in \cite{hs}. However in all
these works, either the choice of H\"older spaces is not appropriate
to the study of kinetic equations or the assumptions on the
coefficients are too strong or the estimate is too
weak. 

In \cite{hs}, the authors use the same natural H\"older spaces
$\cC_\ell^\alpha$ for $\alpha \in (0,1)$ but make other choices for
higher exponents $\alpha$. For instance, following the aforementioned
classical Schauder estimate for parabolic equations, the semi-norm
$[\cdot]_{2+\alpha,Q}$ in \cite{hs} equals the sum of
$[D^2u]_{\alpha,Q}$, $[\partial_t u]_{\alpha,Q}$ \emph{plus} an
additional semi-norm controlling $x$-variations. Such a choice can be
compared the equivalence of norms discussed in
Remark~\ref{rem:equiv}. The authors of \cite{hs} explain their choice
by the fact that the natural Schauder estimate (in the spirit of
Theorem~\ref{thm:variable}) only provides a regularity in the $x$
variable of order $(2+\alpha)/3 <1$ while they aim at reaching
complete smoothing by bootstrap. Note that such bootstrap can
nevertheless be achieved in our spaces $\cC_\ell^\beta$, but this
requires to work with difference derivatives in $x$ of fractional
order $(2+\alpha)/3$ each time the Schauder estimate is applied. Due
to the technical length of this argument when providing full details
and in order to keep this paper concise, we defer this higher
regularity bootstrap on the nonlinear model~\eqref{eq:main} to a
future study, and also note that such bootstrap techniques are being
also implemented for the Boltzmann equation in~\cite{bootstrap}.

We also remark that our choice of norms ---that measures regularity by
estimating the oscillations of a Taylor expansion remainder--- is
related to the proof of the Schauder estimate. Indeed, we adapt the
argument by Safonov \cite{safonov1984} in the parabolic case,
explained in Krylov's book \cite{krylov1996}. In the latter argument,
the oscillation of the remainder of the second-order Taylor expansion
of the solution is shown to decay at rate $r^{2+\alpha}$ in a cylinder
of radius $r$, and a corrector is introduced to the second-order
Taylor polynomial to account for the contribution of the source term
at large distance.  Compared with the parabolic argument
in~\cite{krylov1996}, the main conceptual difference is in the proof
of the \emph{gradient bound}, see Proposition~\ref{prop:grad}.  We
combine Bernstein's method, as in \cite{krylov1996}, with ideas and
techniques borrowed from the hypocoercivity
theory~\cite{villani2009memoir}.

\subsection{Motivation and background for the toy model}

Equation~\eqref{eq:main} describes the evolution of the probability
density function $f$ of particles.  The free transport translates the
fact that the variable $v$ is the velocity of the particle at position
$x$ at time $t$. The operator
$\rho [f] \nabla_v \cdot (\nabla_v f + vf)$ takes into account the
interaction between particles. The diffusion coefficient $\rho [f]$ is
proportional to the total mass of particles lying at $x$ at time $t$:
diffusion is strong in regions where local density is large and weak
in regions where local density is small.  The diffusion operator
$\nabla_v \cdot (\nabla_v f + vf)$ is chosen so that the unique steady
state of this equation is the Gaussian
$\mu(v) = (2\pi)^{-d/2} e^{-|v|^2/2}$.

Even if it has a substantially simpler structure, the nonlinear
kinetic equation~\eqref{eq:main} shares some similarities with the
equation derived by Landau in 1936 \cite{landau1936} as a diffusive
limit, in the regime when grazing collisions dominate, of the
(Maxwell)-Boltzmann equation discovered in 1867--1872
\cite{clerk1867dynamical,boltzmann1872weitere} for describing rarefied
gases. These latter equations are respectively referred to as the
Landau (or Landau-Coulomb) equation and the Boltzmann equation. Landau
derived this equation in order to describe plasmas made up of
electrons and ions interacting by Coulombian forces. The Boltzmann
collision operator can describe long-range interactions less singular
than the Coulomb interactions but does not make sense for the Coulomb
interaction. The Landau collision operator is of the form
$\nabla_v \cdot ( A_f \nabla_vf + \vec b_f f)$ with $A_f = A \star f$
and $\vec b_f = \vec b \star f$ ($\star$ stands for the convolution
with respect to the velocity variable) where $A$ is the matrix
$\frac{1}{|v|} (\mbox{Id} -\frac{v \otimes v}{|v|^2})$ and
$\vec b= - \nabla_v \cdot A$.  The collision operator in
equation~\eqref{eq:main} corresponds to the (much simpler) case where
coefficients are given by $A_f = \rho[f] \mbox{Id}$ and
$B_f = \rho[f] v$.

In the case of the Landau equation, both coefficients are defined by
integral quantities involving the solution. Our simplified toy
model~\eqref{eq:main} replaces these convolutions crudely by their
averages and neglects the issues of the various positive or negative
moments at large velocities. This explains the factor $\rho[f]$. Our
simplified toy model also shares the same Gaussian steady state as the
Landau collision operator.

This simplification respects the principle at the source of
nonlinearity in bilinearity collision operators: that the amount of
collisions at a point is related to the local density of particles.
Note that replacing $\rho[f]$ by another $v$-moment of the solution,
or even having different $v$-moments in front respectively of the
diffusion and drift terms, could most likely be treated by variants of
the method developed in this paper. It is also likely that replacing
$\rho[f]$ by $F(\rho[f])$ where $F:\R_+^* \to \R_+^*$ is a smooth
nonlinear map could be treated by variants of the methods in this
paper.

The model~\eqref{eq:main} was also studied in~\cite{MR2252155} (see
equation (9) in~\cite{MR2252155}, when keeping only mass conservation)
and the authors show how its spatially homogeneous version arise as a
mean-field limit of an $N$-particle Markov process in the spirit of
Kac's process~\cite{MR0084985}. It is also related to the gallery of
nonlinear Fokker-Planck models discussed for instance
in~\cite{chavanis2008nonlinear}.

A recent line of research consists in extending methods from the
elliptic and parabolic theories to kinetic equations. Silvestre in
particular made key progresses on the Boltzmann equation without
cut-off in~\cite{MR3551261}, and together with the first author later
obtained local H\"older estimate in \cite{is} and a Schauder estimate
for a class of kinetic equations with integral fractional diffusion
in~\cite{imbert2018schauder}. In parallel, a similar program was
initiated for the Landau equation in~\cite{gimv}, following up from an
earlier result in~\cite{wz09}. The local H\"older estimate is obtained
in~\cite{wz09,gimv} for essentially bounded solutions, the Harnack
inequality is proved in~\cite{gimv} and some Schauder estimates are
derived in \cite{hs}. The results contained in the present article are
part of this emerging trend in kinetic theory.

\subsection{Strategy of proof for global well-posedness}

We explain here the various ingredients used in the proof of the
global well-posedness of equation~\eqref{eq:main}. The proof proceeds
in 5 steps.
\begin{enumerate}
\item First, the maximum principle implies that if the initial datum
  $f_{in}$ lies between $C_1 \mu$ and $C_2 \mu$, then the
  corresponding solution $f$ to~\eqref{eq:main} satisfies the same
  property: as long as the solution $f(t,x,v)$ is well-defined, we
  have $C_1 \mu (v) \le f(t,x,v) \le C_2 \mu (v)$ for all $(t,x,v)$
  (see Lemma~\ref{l:gaussian}). In particular, this ensures that the
  solution $f$ has fast decay at large velocities and that the
  diffusion coefficient $\rho [f]$ satisfies
  $C_1 \le \rho [f](t,x) \le C_2$ for all $(t,x)$. Therefore, the
  equation satisfies the uniform ellipticity condition in $v$ as
  stated in~\eqref{eq:elliptic}.
\item We deduce from the bound on $\rho [f]$ that the solution $f$
  satisfies an equation of the form
  \[
    (\partial_t + v \cdot \nabla_x ) f = \nabla_v \cdot (A \nabla_v f+
    \vec b f)
  \]
  for a symmetric real matrix $A$ whose eigenvalues all lie in
  $[C_1,C_2]$.  In particular, we can use the local H\"older estimate
  from \cite{wz09,gimv}, see Theorem~\ref{t:gimv} from
  Subsection~\ref{ss:holder}. The decay estimate from Step~1 and the
  H\"older regularity $\cC_\ell^{\alpha_0}$ for some small $\alpha_0$
  are then combined with the Schauder estimate from
  Theorem~\ref{thm:variable} to derive a \emph{higher-order} H\"older
  estimate in $C^{2+\alpha_0}_\ell$ (see
  Proposition~\ref{prop:local}).
\item With such a higher order H\"older estimate at hand, we next
  study how Sobolev norms in $x$ and $v$ grow as time increases and we
  derive a \emph{continuation criterion} (in the same spirit as the
  Beale-Kato-Majda blow-up criterion \cite{MR763762}). We prove 
  then that the blow-up is prevented by the $\cC_\ell^{2+\alpha_0}$
  H\"older estimate from Step~2. This finally yields global
  well-posedness of equation~\eqref{eq:main} in Sobolev spaces.
\end{enumerate}

It is worth mentioning that a conditional global smoothing effect for
the Landau equation with moderately soft potentials has been recently
obtained in \cite{hs} by combining the ingredients listed in Steps 1
to 3 above. Moreover, establishing such a global smoothing effect is
in progress for the Boltzmann equation without cut-off with moderately
soft potentials \cite{bootstrap}. These works however assume \emph{a
  priori} that some quantities such as mass, energy and entropy
densities remain under control along the flow. The \emph{a priori}
assumption is necessary to prove, among other things, that the
equations enjoy some uniform ellipticity and to establish good decay
in the velocity variable. The interest of the toy nonlinear model
\eqref{eq:main} lies in the fact that it is a nontrivial and
physically relevant model for which \emph{unconditional} global
well-posedness can be proven following such a programme. The main
simplification of our model compared with the Landau equation is the
lack of local conservation of momentum and energy; therefore the fluid
dynamics on the local density, momentum and energy fields reduces to
the heat flow in the fluid limit and avoids the difficulties of the
Euler and Navier-Stokes dynamics. The results of this paper hence
provide one more hint that the formation of singularities, if any, in
the Cauchy problem for non-linear kinetic equations is likely to come
from (1) fluid mechanics, or (2) issues with the decay at large
velocities.

\subsection{Perspectives}

We conclude the introduction by mentioning that the well-posedness
result for the toy nonlinear model can be improved in two
directions. First, more general initial data could be considered by
constructing solutions directly in our H\"older spaces rather than
mixing H\"older and Sobolev spaces. Second, we previously mentioned
that $C^\infty$ regularization is expected for positive times by
applying iteratively the Schauder estimates.

\subsection{Organisation of the article}

Section~\ref{sec:prelim} is devoted to the definition of H\"older
spaces.  The Schauder estimate from Theorem~\ref{thm:variable} is
proved in Section~\ref{sec:schauder}. We prove
Theorem~\ref{theo:main-holder} in the final Section~\ref{sec:toy} by
constructing local solutions to the non-linear equation
\eqref{eq:main} in Sobolev spaces and by using the Schauder estimate
to extend these solutions globally in time. 

\subsection{Notation} We collect here the main notations for the
convenience of the reader.

\subsubsection{Euclidian space and torus} The $d$-dimensional
Euclidian space is denoted by $\R^d$ and the $d$-dimensional torus by
$\T^d$.  Throughout this article, the space variable $x$ belongs to
$\R^d$, except in Section~\ref{sec:toy} where $x \in \T^d$.

\subsubsection{Multi-indices} The \emph{order} of $m \in \N^d$ is
$|m| := m_1+\dots+m_d$. Given a vector $x \in \R^d$ and $m \in \N^d$,
we denote $x^m := \prod_{i=1} ^d x_i^{m_i}$.

\subsubsection{Balls and cylinders} $B_r$ denotes the open ball of
$\R^d$ of radius $r$ centered at the origin. $Q_r(z_0)$ denotes a
cylinder in $\R \times \R^d \times \R^d$ centered at $z_0$ of radius
$r$ following the kinetic scaling, see~\eqref{eq:def-cyl}. $Q_r$
simply denotes $Q_r((0,0,0))$. The scaled variable is
$S_r (z):=(r^2t,r^3x,rv)$ for $z=(t,x,v)$,
see~\eqref{eq:scaling}. Radii of cylinders are denoted by $r$. Unless
further constraints are stated, this radius is an arbitrary positive
real number. It is sometimes restricted by the fact that a cylinder
should not leak out of the domain of study of the equation (in
particular in time), sometimes chosen to be $1$ or $2$, or
multiplied by a given constant, e.g. $3/2$ or $K+1$.

\subsubsection{Constants} We use the notation $g_1 \lesssim g_2$ when
there exists a constant $C > 0$ independent of the parameters of
interest such that $g_1 \leq C g_2$. We analogously write
$g_1 \gtrsim g_2$.  We sometimes use the notation
$g_1 \lesssim_{\delta} g_2$ if we want to emphasize that the implicit
constant depends on some parameter $\delta$.

\subsubsection{H\"older spaces and exponents} Given an open set $\cQ$,
$\cC_\ell^\alpha (\cQ)$ denotes the set of $\alpha$-H\"older
continuous functions in $\cQ$, see Definition~\ref{def:holder}. The
subscript $\ell$ refers to ``left'', it is not a parameter.  The
letter $\alpha$ denotes an arbitrary positive exponent.  This exponent
will be fixed to some value $\alpha_0$ in Section~\ref{sec:toy} after
applying the local H\"older estimate from \cite{wz09,gimv} recalled in
Subsubsection~\ref{ss:holder}.

\subsubsection{Kolmogorov operator} The Green function of the operator
$\mathcal L_K := \partial_t + v \cdot \nabla_x-\Delta_v$ is denoted by
$\mathcal{G}$.

\section{H\"older spaces}
\label{sec:prelim}

\subsection{Galilean invariance, scaling and cylinders}

Given $z := (t,x,v) \in \R \times \R^d \times \R^d$ and $r>0$ define
\begin{equation}\label{eq:scaling} 
 S_r( z) := (r^2 t,r^3 x, r v).
\end{equation}

The Galilean invariance of equations~\eqref{eq:linear} and
\eqref{eq:main} is expressed with the non-commutative product
\[
(t_1,x_1,v_1) \circ (t_2,x_2,v_2) := (t_1+t_2,x_1 +x_2 + t_2 v_1,v_1 + v_2)
\]
with inverse element denoted $z^{-1} := (-t,-x+tv,-v)$ for
$z:=(t,x,v)$. 

Given $z_0 \in \R \times \R^d \times \R^d$ and $r >0$, we define the
unit cylinder $Q_1:=Q_1(0)=(-1,0]\times B_1 \times B_1$, then
$Q_r:=Q_r(0) = S_r(Q_1)$ and finally
$Q_r (z_0) = \{ z_0 \circ z: z \in Q_r\}$. This results in
\begin{equation}\label{eq:def-cyl}
  Q_r (z_0) :=  \Big\{ (t,x,v) : t_0 - r^2 < t \le t_0, \, |x-x_0 -
  (t-t_0)v_0| < r^3, \, |v-v_0| < r \Big\}.
\end{equation}

\subsection{The Green function}

Consider the Kolmogorov equation
\begin{equation}\label{eq:kolmogorov}
(\partial_t  + v \cdot \nabla_x) g = \Delta_v g + S
\end{equation}
with a given source term $S$. 
The Green function $\cG$ of the operator
$\mathcal L_K:= \partial_t + v \cdot \nabla_x - \Delta_v$ was computed
by Kolmogorov~\cite{kolm}:
\begin{equation}\label{eq:green}
  \cG(z) = \begin{cases}
    \left(\frac{\sqrt{3}}{2 \pi t^2}\right)^d
    e^{-\frac{3\left|x+\frac{t}2 v\right|^2}{t^3}}  e^{-\frac{|v|^2}{4
        t}}
    & \text{ if } t > 0, \\
    0  & \text{ if } t \le 0.
\end{cases}
\end{equation}

\begin{proposition}[Properties of the Green function]
  \label{prop:pregreen}
  Given $S \in L^\infty(\R \times \R^d \times \R^d)$ with compact
  support in time, the function
  \begin{align*}
    g(t,x,v) & = \int_{\R \times \R^d \times \R^d}
               \cG(\tilde z^{-1} \circ z) S(\tilde z) \dd \tilde t \dd
               \tilde x \dd \tilde v \qquad
               \mbox{ (with } z:=(t,x,v) \mbox{ and } \tilde z :=
               (\tilde t,\tilde x,\tilde v) \mbox{)}\\
             & = \int_{\R \times \R^d \times \R^d} \cG(t-\tilde t,x- \tilde x
               - (t - \tilde t)\tilde v,v-\tilde v) S (\tilde t,\tilde x,\tilde
               v) \dd \tilde t \dd \tilde x \dd \tilde v
               =: (\cG \starkin S)(z)
  \end{align*}
  satisfies \eqref{eq:kolmogorov} in $\R \times \R^d \times \R^d$ (the
  ``kinetic'' convolution $\starkin$ follows the Galilean invariance).

  Moreover, for all
  $z_0 =(t_0,x_0,v_0) \in \R \times \R^d \times \R^d$ and $r >0$
  \[
    \left\|\cG \starkin \un_{Q_r(z_0)}\right\|_{L^\infty(Q_r(z_0))}
    \lesssim_d r^2.
  \]
\end{proposition}

\begin{proof} 
The argument of ~\cite[Lemma~8.4.1,~p.~115]{krylov1996} applies
similarly: given $z \in Q_r(z_0)$, calculate
\begin{align*}
  \cG \starkin \un_{Q_r(z_0)}(z)
  & = \int_{Q_r(z_0)} \cG(t-\tilde t,x- \tilde x
    - (t - \tilde t)\tilde v,v-\tilde v) \dd \tilde z \\
  & = r^2 \int_{Q_1(z_0)} \cG\left(\frac{t}{r^2}-\bar t,\frac{x}{r^3}- \bar x
    - \left(\frac{t}{r^2} - \bar t\right)\bar v,\frac{v}{r}-\bar
    v\right) \dd \bar z \\
  & = r^2 \cG \starkin \un_{Q_1(z_0)}\left(S_\frac1r(z)\right)
\end{align*}
which yields the result.
\end{proof}

\subsection{H\"older spaces}

We now introduce H\"older spaces similar to that
in~\cite{imbert2018schauder}. 
\begin{definition}[H\"older spaces]\label{def:holder}
  The \emph{kinetic degree} of a monomial
  $m (t,x,v) = t^{k_0} \Pi_{i=1}^d x_i^{k_i} v_i^{l_i}$ associated
  with the \emph{partial} degrees $k_0,k_1,\dots,k_d \in \N$ and
  $l_1,\dots,l_d \in \N$ is defined as
  \[
    \degkin m := 2 k_0 + 3 \left(\sum_{i=1}^d k_i\right) + \sum_{i=1}^d
    l_i.
  \]
  In particular, constants have zero kinetic degree.  The kinetic
  degree $\degkin p$ of a polynomial $p \in \mathbb{R}[t,x,v]$ is
  defined as the largest kinetic degree of the monomials $m_j$
  appearing in $p$.

  Given an open set $\cQ \subset \R \times \R^d \times \R^d$ and
  $\beta >0$, we say that a function $g: \cQ \to \R$ is
  $\beta$-H\"older continuous at a point $z_0 \in \cQ$ if there is a
  polynomial $p \in \R [t,x,v]$ with $\degkin p < \beta$ and a
  constant $C>0$ such that
\begin{equation}\label{e:p}
    \forall \ r >0, \quad \|g-
    p\|_{L^\infty(Q_r(z_0) \cap \cQ)} \le C r^{\beta}.
  \end{equation}
  If this property holds true for all $z_0 \in \cQ$, the function $g$
  is $\beta$-H\"older continuous in $\cQ$ and we write
  $g \in \cC_\ell^\beta (\cQ)$.  The smallest constant $C$ such that
  the property \eqref{e:p} holds true for all $z_0 \in \cQ$ is denoted
  by $[g]_{\cC^{\beta}(\cQ)}$.

  The $\cC_\ell^{\beta}$-norm of $g$ is then
  $\|g\|_{\cC_\ell^{\beta}(\cQ)} := \|g\|_{L^\infty(\cQ)} +
  [g]_{\cC_\ell^{\beta}(\cQ)}$.
\end{definition}

\begin{remarks}
  \begin{enumerate}
  \item For $\beta \in (0,1)$, the semi-norm $\cC^\beta_\ell$ is
    equivalent to the standard H\"older semi-norm $\cC^\beta$ for the
    distance $\|(t,x,v)\|=|t|^{\frac12} + |x|^{\frac13} + |v|$. 

  \item For a non-zero integer $k \in \N$, the spaces $C^k_\ell$
    differ from the usual $C^k$ spaces, in that the highest-order
    derivatives are not continuous but merely $L^\infty$. For instance
    $C^1_\ell$ functions are Lipschitz continuous in $v$ but not
    continuously differentiable in $v$.

  \item In \cite{imbert2018schauder}, a Schauder estimate is obtained
    for kinetic equations with an integral collision operator with
    fractional ellipticity in the velocity variable $v$. A parameter
    $s \in (0,1)$, related to the order of differentiation of the
    integral operator, plays a role in the definition of cylinders
    (through scaling) and, in turn, in the definition of kinetic
    degree and H\"older spaces. The present definition corresponds to
    the case $s=1$.

  \item The subscript $\ell$ emphasizes the fact that this definition
    is based on the non-commutative product $\circ$ defined above. The
    symbol $\ell$ stands for ``left'': because of Galilean invariance,
    the kinetic equations considered in \cite{imbert2018schauder} and
    the present work are left invariant with respect to this
    product. See \cite{imbert2018schauder} for further discussion.

  \item In \cite{imbert2018schauder}, H\"older spaces are defined by
    using a kinetic distance $d_\ell$. It is pointed out that this
    distance satisfies
    $\frac14 \|z_2^{-1} \circ z_1\| \le d_\ell (z_1,z_2) \le
    \|z_2^{-1} \circ z_1\|$ where
    $\|(t,x,v)\|=|t|^{\frac12} + |x|^{\frac13} + |v|$. This inequality
    justifies the fact that the Definition~\ref{def:holder} above
    coincides with~\cite[Definition~2.3]{imbert2018schauder} and that
    semi-norms only differ by a factor $4^\beta$.

  \item When $\cQ = \R \times \R^d \times \R^d$, we simply write
    $[\cdot]_{\cC_\ell^{\beta}}$, $\|\cdot\|_{\cC_\ell^{\beta}}$,
    $\|\cdot\|_{L^\infty}$, \textit{etc}.
  \end{enumerate}
\end{remarks}

\subsection{Second order Taylor expansion}

When $\beta =2+\alpha \in (2,3)$, we now prove that the polynomial $p$
realizing the infimum in the $\cC_\ell^\beta$-semi-norm is the Taylor
expansion of kinetic degree $2$:
\begin{equation}
\label{e:taylor}
\mathcal T_{z_0}[g](t,x,v) := g(z_0) + (t-t_0) \left[ \partial_t g + v_0
  \cdot \nabla_x g \right](z_0) + (v-v_0) \cdot \nabla_v g(z_0)
+ \frac12 (v-v_0)^T \cdot D^2 _v g(z_0) \cdot (v-v_0).
\end{equation}
Remark that the linear part in $x$ does not appear since it is of
kinetic degree $3$.

We recall and denote
\begin{align*}
  [g]_{\cC_\ell^{2+\alpha}(\cQ)}
  & = \sup \left\{   \frac{ \inf_{p \in \mathbb P}
    \|g-p \|_{L^\infty(Q_r(z_0) \cap \cQ)}}{r^{2+\alpha}} :
    z_0 \in \cQ, \ r >0  \right\} \\
  [g]_{\cC_{\ell,0}^{2+\alpha}(\cQ)}
  & := \sup \left\{\frac{\|g-\mathcal{T}_{z_0}[g]
    \|_{L^\infty(Q_r(z_0) \cap \cQ)}}{r^{2+\alpha}} :
    z_0 \in \cQ, \ r >0 \right\}
\end{align*}
where $\mathbb P$ denotes the set of polynomials of kinetic degree
smaller than or equal to $2$.
\begin{lemma}
  \label{lem:NN0}
  Given $\alpha \in (0,1)$, there exists $C\in (0,1)$ such that for
  all $g \in \cC_\ell^{2+\alpha}(\cQ)$, the derivatives in
  $\mathcal T_{z_0}[g]$ exist for all $z_0 \in \cQ$ and
  \[
    C [g]_{\cC_{\ell,0} ^{2+\alpha}(\cQ)} \le
    [g]_{\cC_\ell^{2+\alpha}(\cQ)} \le
    [g]_{\cC_{\ell,0}^{2+\alpha}(\cQ)}.
  \]
\end{lemma}
\begin{proof}
  First reduce to $z_0=0$ by the change of variables
  $g^\sharp(z) := g(z_0 \circ z)$. We continue however to simply call
  the function $g$. We start by proving the first inequality. One
  needs to identify the minimizer $p \in \R[t,x,v]$ realising
  $\inf_{p \in\R[t,x,v]} \|g-p \|_{L^\infty(Q_r)}$ in the limit
  $r \to 0^+$. Let $\eps >0$ and consider $r_k = 2^{-k}$ and
  $p_k \in \mathbb P$ s.t.
  $\|g-p_k \|_{L^\infty(Q_{r_k})} \le
  r_k^{2+\alpha}([g]_{\cC_\ell^{2+\alpha}(\cQ)} + \eps)$.  Write
  $p_k(t,v) =: a_k + b_k t + q_k \cdot v + \frac12 v^T M_k v$. By
  subtraction one gets
  \begin{align*}
    \left\| p_{k+1}-p_{k} \right\|_{L^\infty(Q_{r_{k+1}})} \le
    2 r_k^{2+\alpha} \left( [g]_{\cC_\ell^{2+\alpha}(\cQ)}+\eps \right)
  \end{align*}
  which writes in terms of the coefficients
  \begin{align*}
    \left\| \left( a_{k}-a_{k+1}\right) + \left( b_{k} - b_{k+1}
    \right) t + \left( q_{k} - q_{k+1} \right) \cdot v + \frac12 v^T
    \left( M_{k} - M_{k+1} \right) v
    \right\|_{L^\infty(Q_{r_{k+1}})}  \le 
    2 r_{k}^{2+\alpha} \left( [g]_{\cC_\ell^{2+\alpha}(\cQ)} + \eps \right). 
  \end{align*}
  Testing for $t=0$ and $v=0$ gives
  $|a_{k}-a_{k+1}| \lesssim r_{k}^{2+\alpha}
  ([g]_{\cC_\ell^{2+\alpha}(\cQ)}+\eps)$.  Using the latter and
  testing for $v=0$ and $|t|=r_{k+1}^2$ gives
  $|b_{k} - b_{k+1}| \lesssim r_{k}^{\alpha}
  ([g]_{\cC_\ell^{2+\alpha}(\cQ)}+\eps)$.  Testing for $t=0$ and
  summing $v$ and $-v$ with $|v|=r_{k+1}$ in all directions gives
  $|M_{k} - M_{k+1}| \lesssim r_{k}^{\alpha} (
  [g]_{\cC_\ell^{2+\alpha}(\cQ)}+\eps)$.  Finally by difference and
  testing with $t=0$ and all directions of $|v|=r_k$, one gets
  $|q_{k} - q_{k+1}| \lesssim r_{k}^{1+\alpha}
  ([g]_{\cC_\ell^{2+\alpha}(\cQ)}+\eps)$.  This shows that the
  coefficients are converging with
  $|a_{k}-a_{\infty}| \lesssim r_{k}^{2+\alpha}
  ([g]_{\cC_\ell^{2+\alpha}(\cQ)}+\eps)$ and
  $|b_{k} - b_{\infty}| \lesssim r_{k}^{\alpha}
  ([g]_{\cC_\ell^{2+\alpha}(\cQ)}+\eps)$ and
  $|q_{k} - q_{\infty}| \lesssim r_{k}^{1+\alpha}
  ([g]_{\cC_\ell^{2+\alpha}(\cQ)}+\eps)$ and
  $|M_{k} - M_{\infty}| \lesssim r_{k}^{\alpha}
  ([g]_{\cC_\ell^{2+\alpha}(\cQ)}+\eps)$.  These convergences and
  estimates imply that
  $\| g - P_\infty \|_{L^\infty(\cQ_{r_k})} \lesssim r_{k}^{2+\alpha}
  ([g]_{\cC_\ell^{2+\alpha}(\cQ)}+\eps)$ which in turn implies that
  $a_\infty = g(0,0,0)$ and $b_\infty = \partial_t g(0,0,0)$ and
  $q_\infty = \nabla_v g(0,0,0)$ and $M_\infty = D^2 _v g(0,0,0)$ (and
  proves the existence of such derivatives). We thus proved that
  $\|g-\mathcal T_0[g] \|_{L^\infty(Q_{r_k})} \lesssim
  r_{k}^{2+\alpha} ([g]_{\cC_\ell^{2+\alpha}(\cQ)}+\eps)$ where the
  constant does not depend on $k$. This in turn implies that same
  inequality for any $r>0$, with a constant at most multiplied by $2$,
  which concludes the proof since $\eps$ is arbitrarily small.

  The proof of the second inequality
  $[g]_{\cC_\ell^{2+\alpha}(\cQ)} \le [g]_{\cC_{\ell,0}
    ^{2+\alpha}(\cQ)}$ then follows from the existence of the
  derivatives appearing in $\mathcal T_0[g]$ showed in the previous step, and
  the fact that $\mathcal T_{0}[g]$ is of kinetic degree strictly smaller than
  $2+\alpha$.
\end{proof}

\subsection{Interpolation inequalities}

Interpolation inequalities are needed in the proof of the Schauder
estimate.
\begin{lemma}[Interpolation inequalities]\label{l:interpol}
  Let $\Qint, \Qext$ be two cylinders of the form $Q_\rho(z_c)$ and
  $Q_R(z_c)$ with either $\rho < R$ or $\rho = R = +\infty$. There
  exist $C,\beta>0$ depending only on $d$, $\alpha \in (0,1)$ (and
  $R-\rho$ if $R$ is finite) such that for any $\eps \in (0,1)$ and
  any $g \in \cC_\ell^{2+\alpha}(\Qext)$, 
  \begin{align}
    \label{eq:interp1}
    [g]_{\cC^\alpha_\ell (\Qint)}
    & \le \eps [g]_{\cC_\ell^{2+\alpha}(\Qext)}
      + C \eps^{-\beta} \|g\|_{L^\infty(\Qext)} \\
    \label{eq:interp3}
    [\nabla_v g]_{\cC_\ell^{\alpha}(\Qint)}
    & \le \eps [g]_{\cC_\ell^{2+\alpha} (\Qext)}
      + C\eps^{-\beta} \|g\|_{L^\infty (\Qext)}\\
    \label{eq:interp4-bis}
    [ D^2 _v g ]_{\cC_\ell^\alpha(\R^{2d+1})}
    & \le  C [g]_{\cC_\ell^{2+\alpha}(\R^{2d+1})} \\
    \label{eq:interp5-bis}
    [ (\partial_t  + v \cdot \nabla_v) g ]_{\cC_\ell^\alpha(\R^{2d+1})} 
    & \le   C  [g]_{\cC_\ell^{2+\alpha}(\R^{2d+1})}.
  \end{align}
  Note that the two last inequalities are only proved in the case of
  the whole space $\rho=R=\infty$.
\end{lemma}
\begin{remark}
  A similar result was also proved
  in~\cite[Lemma~2.7]{imbert2018schauder} by a different argument and
  in the terminology of \emph{kinetic degree}, see
  Definition~\ref{def:holder}. Denote
  $D=\mbox{Id},\nabla_v, D^2_v,(\partial_t + v \cdot \nabla_x)$ and
  define $k \in \{0,1,2\}$ the kinetic order of the differential
  operator $D$, i.e. $k=0$ for $D=\mbox{Id}$, $k=1$ for $D = \nabla_v$
  and $k=2$ for $D=D^2_v$ and $D=\partial_t + v \cdot \nabla_x$ (this
  is the kinetic degree of the polynomial naturally associated with
  the differential operator). Then given for $\beta \in \{0,\alpha\}$,
  and provided that $k+\beta < 2+\alpha$, we have
  \begin{align}\label{e:interpol-cyril}
    \|Dg\|_{\cC_\ell^\beta(\cQ)} \le \eps^{(2+\alpha)-(k+\beta)} [g]_{\cC_\ell^{2+\alpha}(\cQ)}
    + C\eps^{-(k+\beta)} \|g\|_{L^\infty(\cQ)}
  \end{align}
  where constant $C$ only depends on dimension $d$, $\alpha$ and
  $\|\cdot\|_{L^\infty(\cQ)}$. Observe that the restrictions imposed
  on $k,\beta$ yield the same inequalities as in
  Lemma~\ref{l:interpol}. The two key steps
  in~\cite[Lemma~2.7]{imbert2018schauder} are the proof of the general
  interpolation inequality
  $[g]_{\cC_\ell^{k+\beta}(\cQ)} \le
  \|g\|_{L^\infty(\cQ)}^{1-\frac{k+\beta}{2+\alpha}}
  [g]_{\cC_\ell^{2+\alpha}(\cQ)}^{\frac{k+\beta}{2+\alpha}} + C
  \|g\|_{L^\infty(\cQ)}$ and inequalities relating the H\"older
  semi-norms of derivatives of a given function to its (higher order)
  semi-norm, as in~\eqref{eq:interp4-bis}-\eqref{eq:interp5-bis} but
  more general.
\end{remark}
\begin{remark}
  It is possible to get rid of the condition $\rho < R$ when
  $R<+\infty$ but this requires substantial modifications. In the
  following proof, a global estimate ($\rho=R=+\infty$) is derived and
  a local ``interior'' one is easily obtained by a localization
  procedure. Since reaching $\rho =R<+\infty$ is irrelevant for this
  work and the proof below is different from the one contained in
  \cite{imbert2018schauder}, we believe it can be useful to restrict
  ourselves to this special case.
\end{remark}

\begin{proof}[Proof of Lemma~\ref{l:interpol}]
  It is sufficient to prove the interpolation
  inequalities~\eqref{eq:interp1}-\eqref{eq:interp3} in the case where
  $\Qint = \Qext = \R^{2d+1}$. Indeed, it is then sufficient to apply
  global estimates to $\tilde g = g \phi$ for some cut-off function
  $\phi$ such that $0 \le \phi \le 1$, $\phi =1$ in $\Qint$ and
  $\phi=0$ outside $\Qext$ to get the local ones. Indeed, in the case
  of \eqref{eq:interp1},
  \begin{align*}
    [g]_{\cC^\alpha_\ell (\Qint)}
    & \le [g \phi]_{\cC^\alpha_\ell (\R^{2d+1})}  \\
    & \le \eps [g \phi]_{\cC_\ell^{2+\alpha}(\R^{2d+1})}
      + C \eps^{-\beta} \|g \phi\|_{L^\infty(\R^{2d+1})} \\
    & \le \eps \bigg( [g]_{\cC_\ell^{2+\alpha}(\Qext)}
      + \|g\|_{L^\infty (\Qext)} [\phi]_{\cC_\ell^{2+\alpha}(\R^{2d+1})}\bigg)
      + C \eps^{-\beta} \|g\|_{L^\infty (\Qext)} \\
    & \le \eps [g]_{\cC_\ell^{2+\alpha}(\Qext)}
      + \tilde{C} \eps^{-\beta} \|g\|_{L^\infty(\Qext)}
  \end{align*}
  where in the case $0<\rho<R<+\infty$, the constant
  $\tilde{C} \sim (R-\rho)^{-2-\alpha}$ depends on the difference of
  radii. This localization is however not used in the proof
  of~\eqref{eq:interp4-bis}-\eqref{eq:interp5-bis} because it would
  add a term $\|g\|_{L^\infty(\Qext)}$ on the right hand side, which
  we want to avoid for the application of Remark~\ref{rem:equiv}.

  Since we now work in the whole space $\R^{2d+1}$, we do not specify
  the domain of functions spaces in the remainder of the proof.

  We next notice that a scaling argument allows us to only prove the
  inequalities with $\eps =1$: scaling then shows that
  $\beta = - \frac{\alpha}{2}$ in~\eqref{eq:interp1} and
  $\beta=-1-\alpha$ in~\eqref{eq:interp3}. \medskip

  \noindent {\bf Step~1: $L^\infty$ bounds on derivatives.} One can
  perform the same argument as in the proof of the preceding
  Lemma~\ref{lem:NN0} with $\eps = \| g \|_{L^\infty}$, $a_0=0$,
  $b_0=0$, $M_0=0$ as admissible first terms in the sequence to get
  \begin{align*}
    & |a_{k}-a_{\infty}| \lesssim r_{k}^{2+\alpha}
      \left([g]_{\cC_\ell^{2+\alpha}}+ \| g \|_{L^\infty}\right) \\ 
    & |b_{k} - b_{\infty}| \lesssim r_{k}^{\alpha}
      \left([g]_{\cC_\ell^{2+\alpha}}+ \| g \|_{L^\infty}\right) \\ 
    & |q_{k} - q_{\infty}| \lesssim r_{k}^{1+\alpha}
      \left([g]_{\cC_\ell^{2+\alpha}}+ \| g \|_{L^\infty} \right) \\
    & |M_{k} - M_{\infty}| \lesssim r_{k}^{\alpha}
      \left([g]_{\cC_\ell^{2+\alpha}}+ \| g \|_{L^\infty} \right).
  \end{align*}
  This yields 
  \begin{align*}
    |g(z_0)|, \ |\partial_t g(z_0)+v \cdot \nabla_x g(z_0)|, \
    |\nabla_v g(z_0)|, \ |\nabla_v^2 g(z_0)|
    \lesssim [g]_{\cC_\ell^{2+\alpha}}+ \| g \|_{L^\infty}
  \end{align*}
  for any $z_0 \in \Qext$, which yields
  \begin{align}
    \label{eq:interp2}
    \left\|\nabla_v g\right\|_{L^\infty}     & \lesssim  [g]_{\cC_\ell^{2+\alpha}}
                                                    +  \|g\|_{L^\infty} \\
    \label{eq:interp4}
    \left\| D^2 _v g \right\|_{L^\infty}     & \lesssim  [g]_{\cC_\ell^{2+\alpha}}      +  \|g\|_{L^\infty} \\
    \label{eq:interp5}
    \left\| (\partial_t + v \cdot \nabla_v) g \right\|_{L^\infty}     & \lesssim [g]_{\cC_\ell^{2+\alpha}} +  \|g\|_{L^\infty}.
\end{align}

\noindent  {\bf Step~2. Control of lower-order H\"older norms.}
Using the previous $L^\infty$ bounds and Lemma~\ref{lem:NN0},
  \begin{align*}
    \|g-g(z_0)\|_{L^\infty(Q_r(z_0) )}
    & \lesssim
      \|g-\mathcal T_{z_0}[g]\|_{L^\infty(Q_r(z_0))} \\
    + & 
      r \left( \left\| \partial_t g(z_0)+v \cdot \nabla_x g(z_0)
      \right\|_{L^\infty} + 
      \left\| \nabla_v g(z_0) \right\|_{L^\infty} + 
      \left\|\nabla_v^2 g(z_0)\right\|_{L^\infty} \right) \\
    & \lesssim r^{2+\alpha} [g]_{\cC^{2+\alpha}_\ell} + r \left(
      [g]_{\cC_\ell^{2+\alpha}}+ \| g \|_{L^\infty} \right) 
  \end{align*}
  which yields \eqref{eq:interp1} with $\eps=1$.
\medskip
 
\noindent {\bf Step~3. H\"older regularity of first-order derivatives
  in $v$.}  We are left with proving \eqref{eq:interp3}, that is to
say that for all $z_0 \in \R^{2d+1}$ and $z_1 \in Q_r (z_0)$, we have
\[
|\nabla_v g (z_1) - \nabla_v g (z_0) | \lesssim  r^\alpha \left( [g]_{\cC_\ell^{2+\alpha} } + \|g\|_{L^\infty }\right).
\]
In view of \eqref{eq:interp4}, it is enough to consider $r \in (0,1]$. 

Define for $z \in \R^{2d+1}$ and  $u \in \mathbb S^{d-1}$ and $r \in (0,1]$,
  \[
    \sigma_{r,u}^1[g](z) := \frac{1}{r} \left[ g(z \circ (0,0,ru)) -g(z) \right].
  \]
We make two observations about this function. First there exists
$\theta = \theta (z,r,u) \in (0,1)$ such that
\[
  \sigma_{r,u} ^1[g] (z) = \nabla_v g (z \circ (0,0,\theta r u)) \cdot u .
\]
Second, for $p = \mathcal T_{z_0} g(z)$, we have
\[
\sigma_{r,u}^1 [p](z_1) - \sigma_{r,u}^1 [p](z_0) = r u^T D^2_v g (z_0) (v_1-v_0).
\]

Let $z_0 \in \R^{2d+1}$, $z_1 \in Q_r (z_0)$ and
$u \in \mathbb S^{d-1}$ be fixed, and $\delta >0$ to be chosen later.
There then exists $\tilde z_0 \in Q_{\delta r} (z_0)$ and
$\tilde z_1 \in Q_{\delta r} (z_1)$ such that
\[
  \sigma_{\delta r,u} ^1 [g] (z_i) = \nabla_v g (\tilde z_i) \cdot u, \qquad \text{ for } i=0,1.
\]
Then we can write,  using Lemma~\ref{lem:NN0} to get the last inequality,
\begin{align*}
     \left| (\nabla_v g(z_1) - \nabla_v g(z_0) ) \cdot u\right|
   \le & \left| \nabla_v g(z_1)\cdot u - \sigma _{\delta r,u}^1 [g](z_1) \right| 
  + \left|  \sigma _{\delta r,u}[g](z_1) -  \sigma _{\delta r,u}^1 [p](z_1) \right| \\
  & + \left| \sigma_{\delta r,u}^1 [p](z_1) - \sigma_{\delta r,u}^1 [p](z_0) \right| \\
  & +\left| \sigma_{\delta r,u}^1 [p](z_0) - \sigma_{\delta r,u}^1  [g](z_0) \right| 
    +  \left| \nabla_v g(z_0) \cdot u -\sigma_{\delta r,u}^1 [g](z_0) \right| \\
   \le & \left| \nabla_v g(z_1) - \nabla_v g (\tilde z_1) \right|  +
      \left\| \nabla^2 _v g \right\|_{L^\infty} |v_1-v_0| + 
      \left| \nabla_v g(z_0) -\nabla_v g (\tilde z_0) \right| \\
      & + r^{-1} |(g-p) (z_0 \circ (0,0,\delta r u))| + r^{-1} |(g-p)
        (z_0)| \\ 
      & + r^{-1} |(g-p) (z_1 \circ (0,0,\delta r u))| + r^{-1} |(g-p) (z_1)| \\
      \le & 2 (\delta r)^\alpha  [\nabla_v g]_{\cC_\ell^\alpha} 
      +  r \left\| \nabla^2 _v g \right\|_{L^\infty} + 4 r^{1+\alpha} [g]_{\cC^{2+\alpha}_\ell } .
  \end{align*}
  Using the fact that $r \le 1$ and taking the supremum over $z_0,z_1, u$, we get,
\[
  [\nabla_v g]_{\cC_\ell^\alpha}     \le  2 \delta ^\alpha [\nabla_v g]_{\cC_\ell^\alpha} 
      + C \left( \left\| \nabla^2 _v g \right\|_{L^\infty}  +
        [g]_{\cC^{2+\alpha}_\ell } \right).
\]
We can now pick $\delta >0$ such that $2 \delta^\alpha = \frac12$ and
conclude that \eqref{eq:interp3} holds true.
\medskip

\noindent
{\bf Step~4. H\"older regularity of second-order derivative in $v$.}
We now prove~\eqref{eq:interp4-bis} following a similar strategy as in
the previous step: we prove for all $z_0 \in \R^{2d+1}$ and
$z_1 \in Q_r (z_0)$ with $r \in (0,1]$
\[
  \left| \nabla_v ^2 g (z_1) - \nabla_v^2 g (z_0) \right| \lesssim
  r^\alpha [g]_{\cC_\ell^{2+\alpha}}.
\]

Define for $z \in \R^{2d+1}$ and  $u \in \mathbb S^{d-1}$ and $r \in (0,1]$,
\[
  \sigma^2 _{r,u}[g](z) := \frac{1}{r^2} \left[ g(z \circ (0,0,ru)) +
    g(z \circ (0,0,-ru)) - 2 g(z) \right].
\]
This function satisfies for some $\theta = \theta (z,r,u) \in (-1,1)$
(and recalling $p=\mathcal T_{z_0}[g]$)
\[
  \sigma^2_{r,u} [g] (z) = u^T \nabla_v ^2 g (z \circ (0,0,\theta r
  u)) u   \quad \text{ and } \quad \sigma_{r,u}^2[p](z_1) =
  \sigma_{r,u}^2[p](z_0).
\]

Let $z_0 \in \R^{2d+1}$, $z_1 \in Q_r (z_0)$ and
$u \in \mathbb S^{d-1}$ be fixed, and $\delta >0$ to be chosen later.
There then exists $\tilde z_0 \in Q_{\delta r} (z_0)$ and
$\tilde z_1 \in Q_{\delta r} (z_1)$ such that
\[
\sigma_{\delta r,u}^2 [g] (z_i) = u^T \nabla_v ^2g (\tilde z_i)u, \qquad \text{ for } i=0,1.
\]
Then we can write,  using Lemma~\ref{lem:NN0} to get the last inequality,
\begin{align*}
   & \left| u^T \left( \nabla_v ^2 g(z_1) - \nabla_v ^2 g(z_0) \right)u\right|  
    \le  \left| u^T \nabla_v ^2 g(z_1)  u - \sigma _{\delta r,u}^2[g](z_1) \right| 
    + \left|  \sigma _{\delta r,u}^2 [g](z_1) -  \sigma _{\delta
    r,u}^2 [p](z_1) \right| \\
  & \qquad + \left| \sigma_{\delta r,u}^2 [p](z_1) - \sigma_{\delta r,u}^2[p](z_0) \right| 
   +\left| \sigma_{\delta r,u}^2[p](z_0) - \sigma_{\delta r,u}^2 [g](z_0) \right| 
    +  \left| u^T \nabla_v ^2 g(z_0) u -\sigma_{\delta r,u}^2 [g](z_0) \right| \\
  & \quad \le  \left| \nabla_v ^2 g(z_1) - \nabla_v ^2 g  (\tilde z_1) \right|  +
    + \left| \nabla_v ^2 g(z_0) -\nabla_v ^2 g  (\tilde z_0) \right| \\
  & \qquad + r^{-2} |(g-p) (z_0 \circ (0,0,\delta r u))| + r^{-2} |(g-p)
    (z_0 \circ (0,0,-\delta r u))| + 2 r^{-2} |(g-p) (z_0)| \\ 
  & \qquad  + r^{-2} |(g-p) (z_1 \circ (0,0,\delta r u))| + r^{-2} |(g-p)
    (z_1 \circ (0,0,-\delta r u))| + 2 r^{-2} |(g-p) (z_1)| \\
  & \quad \le  2 (\delta r)^\alpha  [\nabla_v ^2g]_{\cC_\ell^\alpha} 
    +  6 r^{\alpha} [g]_{\cC^{2+\alpha}_\ell } .
\end{align*}
Using the fact that $r \le 1$ and taking the supremum over
$z_0,z_1, u$, we get,
\[
  [\nabla_v ^2 g]_{\cC_\ell^\alpha} \le 2 \delta ^\alpha [\nabla_v ^2 
  g]_{\cC_\ell^\alpha} + C  [g]_{\cC^{2+\alpha}_\ell } 
\]
which proves the inequality by choosing $2 \delta^\alpha = \frac{1}{2}$.
\medskip

\medskip

\noindent {\bf Step~5. H\"older regularity of first-order transport
  derivative.}
The proof of~\eqref{eq:interp5-bis} follows the same strategy: we
prove, denoting $Y:=\partial_t + v \cdot \nabla_x$, for all
$z_0 \in \R^{2d+1}$ and $z_1 \in Q_r (z_0)$ with $r \in (0,1]$
\[
  \left| Y g (z_1) - Y g (z_0) \right| \lesssim
  r^\alpha [g]_{\cC_\ell^{2+\alpha}}.
\]

Define for $z \in \R^{2d+1}$ and $r \in (0,1]$,
\[
  \sigma^3 _{r}[g](z) := \frac{1}{r^2} \left[ g(z \circ (r^2,0,0)) -
    g(z) \right].
\]
This function satisfies for some $\theta = \theta (z,r) \in (0,1)$
\[
  \sigma^3_{r} [g] (z) = Y g (z \circ (\theta r^2,0,0)   \quad \text{ and } \quad \sigma_{r}^3[p](z_1) =
  \sigma_{r}^3[p](z_0).
\]

Let $z_0 \in \R^{2d+1}$, $z_1 \in Q_r (z_0)$, and $\delta >0$ to be
chosen later. Then for some $\tilde z_0 \in Q_{\delta r} (z_0)$, 
$\tilde z_1 \in Q_{\delta r} (z_1)$:
\[
  \sigma_{\delta r,u}^3 [g] (z_i) = Y g (\tilde z_i), \qquad \text{ for } i=0,1.
\]
Then we can write,  using Lemma~\ref{lem:NN0} to get the last inequality,
\begin{align*}
   \left| Y g(z_1) - Y g(z_0) \right| 
  & \le  \left| Y g(z_1)  - \sigma _{\delta r}^3[g](z_1) \right| 
    + \left|  \sigma _{\delta r}^3 [g](z_1) -  \sigma _{\delta
    r}^3 [p](z_1) \right| \\
  & \quad + \left| \sigma_{\delta r}^3 [p](z_1)
    - \sigma_{\delta r}^3[p](z_0) \right| 
   +\left| \sigma_{\delta r}^3[p](z_0)
    - \sigma_{\delta r}^3 [g](z_0) \right| 
    +  \left| Y g(z_0) -\sigma_{\delta r}^3 [g](z_0) \right| \\
  & \le  \left| Y g(z_1) - Y g (\tilde z_1) \right|  +
    \left| Y g(z_0) - Yg  (\tilde z_0) \right| \\
  & \quad + r^{-2} |(g-p) (z_0 \circ (\delta^2 r^2,0,0))| + r^{-2} |(g-p) (z_0)| \\ 
  & \quad  + r^{-2} |(g-p) (z_1 \circ (\delta^2 r^2,0,0))| + r^{-2} |(g-p) (z_1)| \\
  & \le  2 (\delta r)^\alpha  [Y g]_{\cC_\ell^\alpha} 
    + 4 r^{\alpha} [g]_{\cC^{2+\alpha}_\ell}.
\end{align*}
Using the fact that $r \le 1$ and taking the supremum over
$z_0,z_1, u$, we get,
\[
  [Y g]_{\cC_\ell^\alpha} \le 2 \delta ^\alpha [Yg]_{\cC_\ell^\alpha} + C  [g]_{\cC^{2+\alpha}_\ell } 
\]
which proves the inequality by choosing $2 \delta^\alpha = \frac{1}{2}$.
\medskip

  This achieves the proof of the lemma.
\end{proof}


We will see below that \eqref{eq:interp4-bis} and
\eqref{eq:interp5-bis} can be combined with the hypoelliptic estimate
contained in Lemma~\ref{l:hypo} below to derive an equivalent
semi-norm for the kinetic H\"older space $\cC_\ell^{2+\alpha}$,
see Remark~\ref{rem:equiv}.

\subsection{A hypoelliptic estimate}

We investigate here how to recover the fact that a given function $g$
lies in $\cC^{2+\alpha}_\ell$ only knowing that its free transport and
its velocity second order derivatives lie in $\cC_\ell^\alpha$.  We
remark that this fact is (almost) a consequence of the Schauder
estimate contained in Theorem~\ref{thm:variable} (as a matter of fact,
it is closer to the local version of this result, see
Theorem~\ref{thm:variable-loc}). However we do not need here to
control the $L^\infty$ norm of $g$. The proof of the following lemma
illustrates the hypoelliptic structure of the H\"older spaces
$\cC_\ell^\beta$ and is consequently of independent interest. 
\begin{lemma}[Hypoelliptic H\"older estimate]\label{l:hypo}
Let $Q$ be an arbitrary cylinder and $g \in \cC_\ell^{2+\alpha} (\R^{2d+1})$ then 
\begin{equation}
  \label{e:hypo-g}
  [g]_{\cC_\ell^{2+\alpha}(\R^{2d+1})}   \le C ( [(\partial_t + v \cdot \nabla_x) g ]_{\cC_\ell^{\alpha}(\R^{2d+1})}+  [D^2_v g]_{\cC_\ell^{\alpha}(\R^{2d+1})})
\end{equation}
for some constant $C$ only depending on $\alpha$. 
\end{lemma}
\begin{remark}\label{rem:equiv}
  By combining~\eqref{e:hypo-g} with \eqref{eq:interp4-bis} and
  \eqref{eq:interp5-bis}, we deduce that
  $[(\partial_t + v \cdot \nabla_x) g ]_{\cC_\ell^{\alpha}(\R^{2d+1})}
  + [D^2_v g]_{\cC_\ell^{\alpha}(\R^{2d+1})}$ is a semi-norm
  \emph{equivalent} to
  $[\cdot]_{\cC_\ell^{2+\alpha}(\R^{2d+1})}$. In
    order words, although the semi-norm
    $[\cdot]_{\cC_\ell^{2+\alpha}(\R^{2d+1})}$ is defined by measuring
    oscillations around higher-order polynomials, in the whole space
    it can be recovered with the more classical notion of H\"older
    regularity along the highest-order derivatives. We do not know if
    this equivalence is true in a domain, as in
    $[(\partial_t + v \cdot \nabla_x) g ]_{\cC_\ell^{\alpha}(\cQ)} +
    [D^2_v g]_{\cC_\ell ^\alpha (\cQ)} \sim
    [g]_{\cC_\ell^{2+\alpha}(\cQ)}$. 
\end{remark}

\begin{proof}
  It is convenient to write $Y$ for $(\partial_t + v \cdot
  \nabla_x)$. Moreover, since the domain is the whole space
  $\R^{2d+1}$, we do not specify the domain in function spaces. Let
  $z_0 \in \R^{2d+1}$ and $r>0$.  We consider in particular
  $h =(1,u,v) \in Q_1$ and we have $z_0 \circ S_r (h) \in Q_r(z_0)$.

  Proving \eqref{e:hypo-g} for variations along free streaming is
  easy: for all $z_0 \in \R^{2d+1}$ and $r>0$, we have
  \[
    g \left(z_0 \circ (r^2,0,0)\right) - g(z_0) - Yg (z_0) r^2 =
    r^2 \int_0^1 \left[ Y g\left(z_0 \circ (r^2 \theta,0,0)\right)
    -Yg(z_0) \right] \dd \theta 
  \]
  which implies
  \begin{equation}
    \label{eq:taylor-fs}
    \left|g \left(z_0 \circ (r^2,0,0)\right) - g(z_0)-Yg (z_0)
    r^2\right| \le [Yg]_{\cC_\ell^\alpha } r^{2+\alpha}. 
\end{equation}

  The proof of \eqref{e:hypo-g} for variations along the $v$ variable
  is also straightforward: for all $z_0 \in \R^{2d+1}$ and $r>0$,
  \begin{equation}
    \label{eq:taylor-v}
    \left|g(z_0 \circ (0,0,r v)) -g(z_0) - \nabla_v g(z_0)
    \cdot (rv) - \frac{r^2}2 v^TD^2_v g(z_0)v \right| 
    \le \left[D^2_v g\right]_{\cC_\ell^{\alpha}} r^{2+\alpha}. 
  \end{equation}
  
  We then prove \eqref{e:hypo-g} but only for variations along $x$: we
  prove that for all $z_0 \in \R^{2d+1}$ and $u \in B_1$ and $r>0$,
  \begin{equation}
    \label{e:holder-x}
    \left|g\left(z_0 \circ (0,r^3u,0)\right) -g(z_0) \right|
    \le C \left(\left[(\partial_t        + v \cdot \nabla_x) g \right]_{\cC_\ell^{\alpha}} 
      + \left[D_v^2 g\right]_{\cC_\ell^{\alpha}}\right) r^{2+\alpha}
  \end{equation}
  for some constant $C$ only depending on $\alpha$. 
  This is where a hypoelliptic argument is used: $\nabla_x$ is
  realized as the Lie bracket of $\nabla_v$ and $Y$, along
  trajectories.
  \[
    \xymatrixcolsep{7pc} \xymatrixrowsep{4pc} \xymatrix{ z_1
      \ar[r]^{\mbox{{\tiny forward along $v$}}} & z_2
      \ar[r]^{\mbox{{\tiny backward along transport}}}&
      z_3 \ar[d]^{\mbox{\tiny backward along $v$}}\\
      & z_0 \ar[lu]^{\mbox{{\tiny forward along $x$\qquad}}}& z_4
      \ar[l]^{\mbox{{\tiny forward along transport}}} }
  \]
  Let $z_1$ denote $z_0 \circ (0,r^3u,0)$ and
  $z_2 = z_1 \circ (0,0,ru)$ and $z_3 = z_2 \circ (-r^2,0,0)$ and
  $z_4 = z_3 \circ (0,0,-ru)$.  Remark that
  $z_0 = z_4 \circ (r^2,0,0)$. Notice that all points remain in
  $Q_r (z_0)$. We now write
  \begin{align*}
    g(z_0) - g(z_1)
    = & \left[ g\left(z_4 \circ (r^2,0,0)\right) -g (z_4) \right] +
    \left[ g(z_3 \circ (0,0,-ru)) -g(z_3) \right] 
  \\  & + \left[ g\left(z_2 \circ (-r^2,0,0)\right) -g(z_2) \right] -
    \left[ g(z_2\circ (0,0,-ru))-g(z_2)\right].
  \end{align*}
  We now use \eqref{eq:taylor-fs} with $z_0$ replaced successively by
  $z_2$ and $z_4$ and \eqref{eq:taylor-v} with $z_0$ replaced
  successively by $z_3$ and $z_2$, and we get, after summing the four
  resulting inequalities,
  \begin{align}
    \nonumber    |g(z_0) -g(z_1)| \le &  2 ([Yg]_{\cC^\alpha_\ell } + [D^2_v g]_{\cC^\alpha_\ell} ) r^{2+\alpha} + | Yg (z_4)-Yg(z_2)| r^2 \\
    \nonumber                         & + r | (\nabla_vg (z_2)-\nabla_v g (z_3)) \cdot u |
                           + \frac{r^2} 2 | D_v^2 g (z_3)-D_v^2 g(z_2)| \\
    \label{e:almost} |g(z_0) -g(z_1)| \le & 3([Yg]_{\cC^\alpha_\ell } + [D^2_v g]_{\cC^\alpha_\ell} ) r^{2+\alpha}+ r | (\nabla_vg (z_2)-\nabla_v g (z_3)) \cdot u |
  \end{align}
  We now estimate $| (\nabla_vg (z_2)-\nabla_v g (z_3)) \cdot u |$ in terms of
  \[
    [ g ]_{\cC^{2+\alpha}_{\ell,x} } := \sup_{z \in \R^{2d+1}, u \in B_1} \frac{|g(z \circ (0,0,ru)) - g(z)|}{r^{2+\alpha}}.
  \]
  We are going to prove that for all $R>0$, 
  \begin{equation}
    \label{e:dvx}
    | (\nabla_vg (z_2)-\nabla_v g (z_3)) \cdot u | \le  [ D^2_v g ]_{\cC^\alpha_\ell } (2R^{1+\alpha}+Rr^\alpha)  +  3 [Yg]_{\cC^\alpha_\ell } R^{-1} r^{2+\alpha}
    + [ g ]_{\cC^{2+\alpha}_{\ell,x} } R^{-1} (r^2 R)^{\frac{2+\alpha}3}.
  \end{equation}
  In order to derive this estimate, we approximate
  $\nabla_v g (z) \cdot u$ by the finite difference
  $\sigma^3_{R,u}[g](z)$ we already used above. Recall that
  \[ \sigma^3_{R,u} [g](z) = \frac{|g(z \circ (0,0,Ru)) - g(z)|}{R}.\]
  Remark that \eqref{eq:taylor-v} implies 
  \begin{equation}
    \label{e:1}
    |\sigma^3_{R,u} [g] (z) - \nabla_v g (z) \cdot u - \frac{R}2  D^2_v g (z) u \cdot u  | \le  [ D^2_v g ]_{\cC^\alpha_\ell } R^{1+\alpha} .
  \end{equation}
  Moreover, recalling that $z_3 = z_2 \circ (-r^2,0,0)$ and using \eqref{eq:taylor-fs} twice, 
  \begin{align}
 \nonumber    R |\sigma^3_{R,u} [g](z_2) - \sigma^3_{R,u} [g](z_3) | & =  | g(z_3) - g(z_2) - Yg (z_2) (-r^2) | \\
\nonumber                                                           &  + |g(z_2 \circ (0,0,Ru ) \circ (-r^2,0,0)) - g(z_2 \circ (0,0,Ru)) - Yg (z_2 \circ (0,0,Ru))(-r^2) | \\
 \nonumber                                                          & + r^2 |Yg (z_2) - Yg (z_2 \circ (0,0,Ru)| \\
 \nonumber                                                   & + | g (z_2 \circ (0,0,Ru) \circ (-r^2, 0,0) - g( z_3 \circ (0,0,Ru)|\\
    \nonumber                                           & \le 3 [Yg]_{\cC^\alpha_\ell } r^{2+\alpha} \\
    \nonumber                                   & + | g(t_2 -r^2,x_2 -r^2 (v_2 + Ru), v_2 + Ru) - g(t_2-r^2, x_2 -r^2 v_2, v_2 + Ru) |\\
    \label{e:2}                          & \le 3 [Yg]_{\cC^\alpha_\ell } r^{2+\alpha} + [ g ]_{\cC^{2+\alpha}_{\ell,x} } (r^2 R)^{\frac{2+\alpha}3}.
  \end{align}
Combining \eqref{e:1} and \eqref{e:2} yields \eqref{e:dvx}. 
  
We now combine \eqref{e:almost} and \eqref{e:dvx} and get
\begin{align*}
  r^{-2-\alpha}|g(z_0)-g(z_1)|  &\le  3([Yg]_{\cC^\alpha_\ell } + [D^2_v g]_{\cC^\alpha_\ell} ) +   [ D^2_v g ]_{\cC^\alpha_\ell } (2 (R/r)^{1+\alpha} + (R/r)) \\
  &+  3 [Yg]_{\cC^\alpha_\ell } (r/R)
    + [ g ]_{\cC^{2+\alpha}_{\ell,x} } (r/R)^{1-\alpha}.
\end{align*}
We now pick $R = \delta^{-1} r$ with $\delta^{1-\alpha} = 1/2$ and get
\eqref{e:holder-x} for some constant only depending on $\alpha$. This
concludes the proof.
\end{proof}

\section{The  Schauder estimate} 
\label{sec:schauder}

This section is devoted to the proof of the Schauder estimate
(Theorem~\ref{thm:variable}). The proof proceeds in mainly two steps:
in the first step, the matrix $A$ is constant. The estimate is first
obtained when $A$ is the identity matrix (Theorem~\ref{thm:constant}),
then for an arbitray constant diffusion matrix
(Corollary~\ref{cor:constant}).  Then the estimate is established for
variable coefficient by the procedure of freezing coefficients thanks
to interpolation inequalities.

\subsection{A gradient bound for the Kolmogorov equation}

We first establish a gradient bound for solutions of \eqref{eq:linear}
when $A$ is the identity matrix and when there is no lower order terms
($b=0$, $c=0$). The equation is then reduced to \eqref{eq:kolmogorov}.
Recall that $Q_1= (-1,0] \times B_1 \times B_1$ is the cylinder of
radius $1$.

\begin{proposition}[Gradient bound]\label{prop:grad}
  Given $S \in W^{1,\infty} (Q_1)$, consider $g$ solution to
  \eqref{eq:kolmogorov} in $Q_1$, i.e.
  \[
    (\partial_t + v \cdot \nabla_x )g = \Delta_v g + S .
  \]
  Then
  \[
    \forall \, i =1,\dots,d, \quad
    |\partial_{x_i} g(0,0,0)| + |\partial_{v_i} g(0,0,0)| \lesssim_{d}
    \|g\|_{L^\infty(Q_1)}+\|S\|_{L^\infty(Q_1)} +\|\partial_{x_i}
    S\|_{L^\infty(Q_1)}+\|\partial_{v_i} S\|_{L^\infty(Q_1)}.
  \]
\end{proposition}
\begin{remark}
  See also \cite{guillin2012,baudoin2017} for related gradient
  estimates.
\end{remark}
\begin{proof}
  We use Bernstein's method as Krylov does in \cite{krylov1996} in the
  elliptic-parabolic case, combined with methods from hypocoercivity
  theory (see for instance \cite{villani2009memoir}) in order to
  control the full $(x,v)$-gradient of the solution: see the
  construction of the quadratic form $w$ in $\partial_{x_i} g$ and
  $\partial_{v_i} g$ below.

  Denote the Kolmogorov operator
  $\mathcal L_K g := \partial_t g + v \cdot \nabla_x g - \Delta_v g$
  and compute the following defaults of distributivity of the operator
  (reminiscent of the so-called $\Gamma$-calculus~\cite{MR889476})
  \begin{equation}\label{eq:square}
    \mathcal L_K(g_1g_2) = g_1 \mathcal L_K g_2 + g_2 \mathcal L_K g_1
    - 2 \nabla_v g_1 \cdot \nabla_v g_2,\qquad \mathcal L_K (g^2)
    = 2 g \mathcal L_K g - 2 \left| \nabla_v g \right|^2. 
  \end{equation}
  
  Consider a cut-off function $0 \le \zeta \in C^\infty$ such that
  $\zeta^{1/2} \in C^\infty$, with support in
  $(-1,0] \times B_1 \times B_1$ and such that $\zeta(0,0,0)=1$. In
  order to estimate the gradient of $g$ in $x$ and $v$ at the origin,
  it is enough to find $\nu_0, \nu_1>0$ and
  $0 < \mathfrak a \le \mathfrak b$ and
  $0 < \mathfrak c < \mathfrak a \mathfrak b$ such that, for any
  $i \in \{1,\dots,d\}$,
  \[
    w = \nu_0 g^2 - \nu_1 t + \left[ \mathfrak a^2 \zeta^4
      (\partial_{x_i}g)^2 + \mathfrak c \zeta^3 (\partial_{x_i} g
      )(\partial_{v_i}g) + \mathfrak b^2 \zeta^2
      (\partial_{v_i}g)^2\right]
  \]
  satisfies 
  \[
    - \mathcal L_K w \ge 0.
  \]
  Indeed, the maximum principle for
  parabolic equations then implies that
  $\sup_{Q_1} w = \sup_{\partial_p Q_1} w$ where
  $\partial_p Q_1 = \{-1\} \times B_1 \times B_1 \cup [-1,0] \times
  S_1 \times S_1$ (parabolic boundary). Since $\zeta \equiv 0$ in
  $\partial_p Q_1$ and $\zeta (0,0,0) =1$, we get
  \begin{align}
    \nonumber
    & \mathfrak a^2 \left[ (\partial_{x_i}g)^2 +  (\partial_{v_i}g)^2
      \right](0,0,0) \le \left[ \mathfrak a^2 (\partial_{x_i}g)^2 +  \mathfrak b^2 (\partial_{v_i}g)^2
      \right](0,0,0)+ 2 \nu_0 g^2(0,0,0) \\ \nonumber
    & \le 2 \left[ \mathfrak a^2 (\partial_{x_i}g)^2 +
      \mathfrak c (\partial_{x_i} g \partial_{v_i}g) + \mathfrak b^2 (\partial_{v_i}g)^2
      \right](0,0,0) + 2 \nu_0 g^2(0,0,0) \\ \nonumber
    & 
      \le 2 \left[ \mathfrak a^2 \zeta^4(\partial_{x_i}g)^2 + \mathfrak c \zeta^3(\partial_{x_i} 
      g )( \partial_{v_i}g) + \mathfrak b^2 \zeta^2(\partial_{v_i}g)^2 \right](0,0,0) +
      2 \nu_0 g^2(0,0,0) \\
    & \le 2 w(0,0,0) \le 2 \left(\nu_0 \sup_{Q_1} g^2 + \nu_1\right).
      \label{eq:cs}
  \end{align}

  (I) Observe first that $-\mathcal L_K (-\nu_1t) = \nu_1$. 
  \smallskip
  
  (II) Compute second
  $-\mathcal L_K (\nu_0 g^2)$ using \eqref{eq:square}
  \begin{equation}\label{eq:mu}
    - \mathcal L_K(\nu_0 g^2) = 2\nu_0 |\nabla_v g|^2 - 2 S \nu_0 g.
  \end{equation}
\smallskip

  (III) Compute third $- \mathcal L_K(\zeta^4 (\partial_{x_i}g)^2)$ using
  that $\mathcal L_K(\partial_{x_i}g)=\partial_{x_i}S$ and
  \eqref{eq:square}
  \begin{align}\label{eq:zeta}
    - \mathcal L_K(\zeta^4(\partial_{x_i}g)^2)
    & =  2 \zeta^4 |\nabla_v
      \partial_{x_i}g|^2 - (\partial_{x_i}g)^2 \mathcal L_K (\zeta^4)
      + 2 \nabla_v (\zeta^4) \cdot
      \nabla_v (\partial_{x_i} g)^2 - 2 \zeta^4 \partial_{x_i}g \partial_{x_i}S\nonumber \\
    & \ge  \zeta^4 |\nabla_v
      \partial_{x_i}g|^2 +
      (\partial_{x_i}g)^2
      \left[ - \mathcal L_K (\zeta^4) -
      4 \zeta^{-4} |\nabla_v (\zeta^4)|^2 \right]
      - \zeta^3 (\partial_{x_i}g)^2 - \zeta^5 (\partial_{x_i}S)^2.
  \end{align}
\smallskip

  (IV) Compute fourth, for some $\epsilon_1>0$, the term
  $- \mathcal L_K(\zeta^2(\partial_{v_i}g)^2)$ using
  $\mathcal L_K (\partial_{v_i} g) = \partial_{v_i}S -
  \partial_{x_i}g$ and~\eqref{eq:square}:
  \begin{align}
    - \mathcal L_K(\zeta^2 (\partial_{v_i}g)^2)
    & = 2 \zeta^2 |\nabla_v \partial_{v_i} g|^2 + 2 \nabla_v
      (\zeta^2) \cdot \nabla_v  (\partial_{v_i} g)^2 -
      (\partial_{v_i} g)^2 \mathcal L_K  (\zeta^2) + 2 \zeta^2
      \partial_{v_i} g \partial_{x_i} g
      -2 \zeta^2 \partial_{v_i}g \partial_{v_i} S \nonumber \\
    & \ge \zeta^2
      |\nabla_v \partial_{v_i} g|^2 + (\partial_{v_i} g)^2 \left[- 1 -
      \epsilon_1 ^{-1} \zeta -
      \mathcal L_K (\zeta^2)  - 2 \zeta^{-2}  |\nabla_v \zeta|^2\right]
      -  \epsilon_1 \zeta^3 (\partial_{x_i}g)^2 - \zeta^4
      (\partial_{v_i} S)^2.
    \label{eq:barbarzeta}
  \end{align}
\smallskip

  (V) Compute fifth, for some $\epsilon_2>0$, the term
  $-\mathcal L_K [\zeta^3 (\partial_{x_i}g) (\partial_{v_i} g)]$, with
  the intermediate step
  \[
    - \mathcal L_K [(\partial_{x_i} g)( \partial_{v_i} g)] =
    (\partial_{x_i} g)^2 + 2 \nabla_v \partial_{x_i} g \cdot \nabla_v
    \partial_{v_i} g -\partial_{x_i}g \partial_{v_i}S - \partial_{v_i}
    g \partial_{x_i}S ,
  \]
  \begin{align}
    \nonumber  - \mathcal L_K [\zeta^3 (\partial_{x_i} g)(\partial_{v_i} g)]  =
    &\zeta^3  \left[ (\partial_{x_i}
      g)^2 + 2 \nabla_v \partial_{x_i} g \cdot \nabla_v \partial_{v_i} g
      \right]
      - (\partial_{x_i}g )(\partial_{v_i} g) \mathcal L_K (\zeta^3) \\
    &+ 2 \nabla_v (\zeta^3) \cdot \nabla_v [( \partial_{x_i}
      g)(\partial_{v_i}g)] \nonumber - \zeta^3 \partial_{x_i}g
      \partial_{v_i}S - \zeta^3 \partial_{v_i}g \partial_{x_i}S\nonumber \\
    \ge
    & \frac12 \zeta^3  (\partial_{x_i} g)^2  - \epsilon_2 \zeta^4 |\nabla_v
      \partial_{x_i}g|^2
      -  \epsilon_2^{-1} \zeta^2 |\nabla_v \partial_{v_i} g |^2 
      \label{eq:barzeta}  -(\partial_{x_i}g )(\partial_{v_i} g) \mathcal L_K
      (\zeta^3) \\
    &   + 2 \nabla_v (\zeta^3) \cdot \nabla_v [( \partial_{x_i}
      g)] (\partial_{v_i}g)
      + 2 \nabla_v (\zeta^3) \cdot \nabla_v [(\partial_{v_i}g)] ( \partial_{x_i}
      g) \nonumber \\
    & - \frac12 \zeta^3 (\partial_{v_i}S)^2
      - \frac12 \zeta^3 (\partial_{v_i}g)^2 - \frac12 \zeta^3
      (\partial_{x_i}S)^2.
      \nonumber
  \end{align}
\smallskip

  To clean a little the calculations, observe that
  \begin{itemize}
  \item[(1)] error terms involving $S$ are controlled by choosing
    $\nu_1$ larger enough, i.e. larger than a multiple of
    $\|S\|_{L^\infty(Q_1)} +\|\partial_{x_i}
    S\|_{L^\infty(Q_1)}+\|\partial_{v_i} S\|_{L^\infty(Q_1)}$,
  \item[(2)] terms involving $|\nabla_v g|^2$ or
    $\nabla_v g \cdot \nabla_x g$ are controlled by choosing
    $\nu_0$ large enough thanks to~\eqref{eq:mu}, with $g$ controlled
    by its sup norm, 
  \item[(3)] Equation~\eqref{eq:mu} is ``free'' (i.e. not
    involved in any constant dependency) as well as
    equation~\eqref{eq:barbarzeta} by choosing $\epsilon_1$ small
    enough so that the term $-\epsilon_1 \zeta^3 (\partial_{x_i} g)^3$
    is compensated by the positive term in~\eqref{eq:barzeta}, 
  \item[(4)] Equation~\eqref{eq:zeta} has an error term of the form
    $-O(1) \zeta^3 (\partial_{x_i} g)^2$ that is compensated by the
    positive term in~\eqref{eq:barzeta} again: we use
    $|- \mathcal L_K (\zeta^4) - 2 \zeta^{-4} |\nabla_v (\zeta^4)|^2 |
    \lesssim \zeta^3$ (remember that $\zeta^{\frac12} \in C^\infty$).    
  \item[(5)] In the last equation~\eqref{eq:barzeta} we also split
    \begin{align*}
      & - (\partial_{x_i}g )(\partial_{v_i} g) \mathcal L_K (\zeta^3)
      \lesssim \frac18 (\partial_{x_i} g)^2 + O(1) (\partial_{v_i}
      g)^2,\\
      & 2 \nabla_v (\zeta^3) \cdot \nabla_v [( \partial_{x_i} g)]
      (\partial_{v_i}g) \lesssim \epsilon_3 \zeta^4 |\nabla_v [(
      \partial_{x_i} g)]|^2 + \epsilon_3 ^{-1} O(1) (\partial_{v_i}g)
      ^2 \\
    & 2 \nabla_v (\zeta^3) \cdot \nabla_v [(\partial_{v_i}g)] (
    \partial_{x_i} g) \lesssim O(1) \zeta^2 |\nabla_v
    [(\partial_{v_i}g)] ( \partial_{x_i} g)|^2 + \frac{1}{4}
    \zeta^3(\partial_{x_i} g)^2,
    \end{align*}
     where we have used again $\zeta^{\frac12} \in C^\infty$, in the form
    $|\nabla \zeta| \lesssim \zeta^{1/2}$.
  \end{itemize}
  \medskip
  
  These considerations result in the following calculations:  
  \begin{align*}
      - \mathcal L_K w \ge 
    & 2 \nu_0 |\nabla_v g|^2 + \nu_1 + \mathfrak a^2 \zeta^4 |\nabla_v
      \partial_{x_i}g|^2 - \mathfrak a^2 O(1) \zeta^3 (\partial_{x_i}g)^2 - \mathfrak a^2
      O(1) (\partial_{x_i}S)^2 \\
    &             + \mathfrak b^2 \zeta^2
      |\nabla_v \partial_{v_i} g|^2 - \mathfrak b^2 O(1) (\partial_{v_i} g)^2 
      -  \mathfrak b^2 \epsilon_1 \zeta^3 (\partial_{x_i}g)^2 - \mathfrak b^2 O(1)  (\partial_{v_i}
      S)^2 \\
    &   + \frac{\mathfrak c}4 \zeta^3  (\partial_{x_i} g)^2  
    - \mathfrak c (\epsilon_2 +
      \epsilon_3) \zeta^4 |\nabla_v
      \partial_{x_i}g|^2 -  \mathfrak c O(1) \zeta^2 
      |\nabla_v \partial_{v_i} g |^2 
      - \mathfrak c O(1) (\partial_{x_i,v_i}S)^2 
      - \mathfrak c O(1)  (\partial_{v_i}g)^2.
  \end{align*}
  We finally choose
  \begin{enumerate}
  \item $\mathfrak a=1$,
  \item $\mathfrak c$ large enough so that the first term in the third
    line controls the fourth term in the first line,
  \item $\epsilon_2$ and $\epsilon_3$ small enough so that the second
    term of the third line is controlled by the second term in the
    right hand side of the first line,
  \item $\mathfrak b$ large enough so that the third term in the
    fourth line is controlled by the first term in the third line and
    $\mathfrak a \mathfrak b > \mathfrak c$ so that the quadratic form
    is strictly positive,
  \item $\epsilon_1$ small enough so that the third term in the third
    line is controlled by the first term in the third line,
  \item finally $\nu_0$ and $\nu_1$ large enough to control all the
    $v$-derivatives of $g$ and all the derivatives of $S$. Notice that
    $\nu_1 = O(1)( \| g \|_{L^\infty(Q_1)} + \|S\|_{L^\infty(Q_1)} ^2
    + \|\partial_{x_i} S\|_{L^\infty(Q_1)} ^2+ \|\partial_{v_i}
    S\|_{L^\infty(Q_1)} ^2)$.
  \end{enumerate}
  The conditions on the coefficients are:
  \begin{align*}
    \epsilon_2, \epsilon_3 <<1, \quad \mathfrak c >> 1, \quad
    \mathfrak b^2 >> \mathfrak c, \quad
    \epsilon_1 << \frac{\mathfrak c}{\mathfrak b^2}, \quad
    \nu_0, \nu_1 >> \mathfrak b^2, \mathfrak c
  \end{align*}
  which is compatible with all requirements and has solutions. This
  proves that $-\mathcal L_K \omega \ge 0$ and the desired inequality
  is thus obtained from \eqref{eq:cs} and the choice of $\nu_1$ we
  made above, which concludes the proof.
\end{proof}

A direct consequence of the gradient estimate from
Proposition~\ref{prop:grad} is a bound of derivatives of arbitrary
order in any cylinder of radius $r$. We recall that
$Q_r = (-r^2,0] \times B_{r^3} \times B_r$.
\begin{corollary}[Bounds on arbitrary derivatives around the
  origin]\label{cor:deriv}
  Given $k \in \N$ and $r>0$, there exists a constant $C$ depending on
  dimension $d$ and $k$ such that any solution of
  \eqref{eq:kolmogorov} in $Q_r$ with zero source term $S \equiv 0$,
  i.e. $(\partial_t + v \cdot \nabla_x) g = \Delta_v g$ in $Q_r$,
  satisfies for all $n \in \N$ and $\alpha,\beta \in \N^d$
  with $|\beta|=k$,
  \[
    |\partial_t^n D^\alpha_x D^\beta_v g(0,0,0)| \le \frac{C
      \|g\|_{L^\infty(Q_r)}}{r^{2n +3|\alpha|+|\beta|}}.
  \]
\end{corollary}
\begin{remark}
  Remark that $2n+3|\alpha|+|\beta|$ is the kinetic degree of the
  polynomial associated with $\partial_t^n D^\alpha_x D^\beta_v$.
\end{remark}
\begin{proof}
  We reduce to the case $r=1$ by rescaling: the function
  $g_r (t,x,v) = g(r^2 t, r^3 x, r v)$ is a solution of
  \eqref{eq:kolmogorov} in $Q_1$. If the result is true for $r=1$,
  then we get the desired estimate for arbitrary $r$'s. We then first
  treat the case $n=0$ and argue by induction on $|\beta|$.
  Proposition~\ref{prop:grad} yields the result for $|\beta|\le 1$
  since $D^\alpha_x g$ solves \eqref{eq:kolmogorov} with $S \equiv 0$
  for an arbitrary multi-index $\alpha$. Assuming the result true for
  $n=0$, any $\alpha$ and $|\beta| \le k$, remark that
  $D^\alpha_x D_v ^\beta g$ solves \eqref{eq:kolmogorov} with
  $S =\sum_{i=1,\dots,d \, | \, \beta_i \ge 1} D^{\alpha+\delta_i} _x
  D^{\beta-\delta_i}_ v g$. Consider $\beta \in \N^d$ with
  $|\beta|=k+1$; the previous step yields controls of
  $\partial_{x_j} S$ and $\partial_{v_j} S$ for the source term $S$ in
  the equation for $D^\alpha_x D_v ^\beta
  g$. Proposition~\ref{prop:grad} then gives the control
  $\partial_{x_i,v_i} D^\alpha_x D_v ^\beta g(0,0,0)$ which completes
  the induction. We finally get the result for an arbitrary $n \ge 1$
  by remarking that the equation allows us to control any time
  derivatives by space and velocity derivatives.
\end{proof}

\subsection{Proof of Schauder estimates}

With such bounds on derivatives at hand, we can turn to the proof of
the Schauder estimate for the Kolmogorov equation, that is to say for
equation~\eqref{eq:linear} with $A$ replaced with the identity matrix
and with no lower order terms ($b=0$, $c=0$),
see~\eqref{eq:kolmogorov}. A change of variables will then yield the
result for any constant matrix $A$ satisfying the ellipticity
condition~\eqref{eq:elliptic}. Finally we shall classically
approximate the coefficients locally by constants to treat the general
case.

\subsubsection{The core estimate}

The proof of the Schauder estimate for the Kolmogorov equation follows
the argument proposed by Safonov \cite{safonov1984}, as explained in
\cite{krylov1996}.
\begin{theorem}[The Schauder estimate for the Kolmogorov
  equation]\label{thm:constant}
  Given $\alpha \in (0,1)$ and $S \in \cC_\ell^\alpha$, let
  $g \in \cC_\ell^{2+\alpha}$ be a solution to \eqref{eq:kolmogorov},
  i.e. $(\partial_t + v \cdot \nabla_x) g = \Delta_v g + S$ in
  $\R \times \R^d \times \R^d$. Then
  \[
    [g]_{\cC_\ell^{2+\alpha}} \lesssim_{d,\alpha}
    [S]_{\cC_\ell^\alpha}.
  \]
\end{theorem}
\begin{remark}
  This theorem is the hypoelliptic counterpart
  of~\cite[Theorem~8.6.1~\&~Lemma~8.7.1]{krylov1996}.
\end{remark}
\begin{proof}
  We first reduce to the case where
  $g \in C^\infty_c (\R \times \R^d \times \R^d)$ by mollification and
  truncation (as for instance in the proof of
  \cite[Lemma~8.7.1,~p.~122]{krylov1996}). We then reduce to the base
  point $z_0=0$ by considering the change of unknown
  $g^\sharp(z) := g(z_0 \circ z)$ (we however keep on calling the
  unknown $g$). Given $r>0$ and $K \ge 1$ to be chosen later, consider
  $Q_{(K+1)r}$ and a cut-off function $\zeta \in C^\infty_c$ such that
  $\zeta \equiv 1$ in $Q_{(K+1)r}$. Recall again the Kolmogorov
  operator $\mathcal L_K := \partial_t + v \cdot \nabla_x - \Delta_v$,
  and define $\bar S := \mathcal L_K (\zeta \mathcal T_{0}[g] )$ with
  the Taylor polynomial $\mathcal T_{z_0}[g]$ of $g$ at $(0,0,0)$,
  defined in~\eqref{e:taylor}.
  
  Decompose in $Q_{(K+1)r}$ (where $\zeta =1$):
  \[
    g - \mathcal T_{0}[g] = g  - \zeta
    \mathcal T_{0}[g] = \cG \starkin (S -\bar S) = h_1 + h_2 \quad \mbox{ with }
    \quad
    \begin{cases} h_1 := \cG \starkin \left[ (S- \bar S)
        \un_{Q_{(K+1)r}} \right], \\[2mm]
      h_2 := \cG \starkin \left[ (S-\bar S) \un_{Q^c_{(K+1)r}} \right]
    \end{cases}
  \]
  where
  $Q^c_{(K+1)r} = \R \times \R^d \times \R^d \setminus Q_{(K+1)r}$ and
  $\cG$ is the Green function studied in
  Proposition~\ref{prop:pregreen}.

  Observe that $h_1$ is the solution to
  $\partial_t h_1 + v \cdot \nabla_x h_1 -\Delta_v h_1 = S-S(0,0,0)$
  since
  \[
    \bar S(z) = \mathcal L_K(\zeta \mathcal T_0[g])(z) = \mathcal
    L_K(\mathcal T_0[g])(z)= \left(\partial_t + v \cdot \nabla_x
      g\right)(0,0,0) - \Delta_v g(0,0,0) = \mbox{cst} = S(0,0,0)
  \]
  for $z \in Q_{(K+1)r}$, and $h_2$ is the solution to
  $\partial_t h_2 + v \cdot \nabla_x h_2- \Delta_v h_2 = 0$ in
  $Q_{(K+1)r}$.

  We next estimate 
  \begin{equation}
    \label{eq:start}
    \| g - \mathcal T_{0}[g] - \mathcal T_{0}[h_2] \|_{L^\infty(Q_r)} \le \|
    h_2 - \mathcal T_{0}[h_2] \|_{L^\infty(Q_r)} + \| h_1\|_{L^\infty(Q_r)}.
  \end{equation}
  Using Proposition~\ref{prop:pregreen} and $S \in \cC^\alpha_\ell$,
  we get
  \begin{equation}
    \label{eq:g}
    \| h_1 \|_{L^\infty(Q_r)} \lesssim
    (K+1)^{2+\alpha} r^{2+\alpha} [S]_{\cC_\ell^{\alpha}(Q_{(K+1)r})}.
  \end{equation}
  Now for $z = (r^2 t, r^3 x, rv) \in Q_r$ with $(t,x,v) \in Q_1$.
  There exists $\theta_1, \theta_2, \theta_3 \in (0,1)$ such that
  \begin{align*}
    h_2(z) =
    & h_2(r^2 t, r^3 x , rv) \\
    = & h_2(r^2 t, 0, rv) + (r^3 x) \cdot \nabla_x h_2 (r^2t,r^3 \theta_1 x
        ,r v) \\
    = & h_2(0, 0, rv) + (r^3 x) \cdot \nabla_x h_2 (r^2t,r^3 \theta_1 x
        ,r v)
        + r^2t \partial_t h_2 ( \theta_2 r^2 t, 0,r v) \\
    = & h_2(0,0,0) + (r^3 x) \cdot \nabla_x h_2 (r^2t,r^3 \theta_1 x
        ,r v)
        + r^2t \partial_t h_2 ( \theta_2 r^2 t, 0,r v) \\
    & \quad 
      +  \nabla_v h_2(0,0,0) \cdot (rv)
      + \frac12   (rv) ^T \cdot D^2_v h_2(0,0, r \theta_3 v) \cdot (rv).
  \end{align*}
  As a consequence
  \begin{align*}
    \| h_2 - \mathcal T_{0}[h_2]\|_{L^\infty(Q_r)}
    & \le r^2 \Big[ r \|\nabla_x h_2\|_{L^\infty(Q_{(K+1)r})} + 
      r^2 \|\partial^2_t h_2\|_{L^\infty(Q_{(K+1)r})} \\
    & \qquad + r \|\partial_t \nabla_v h_2 \|_{L^\infty(Q_{(K+1)r})}
      + r \| D^3_v h_2 \|_{L^\infty(Q_{(K+1)r})} \Big].
  \end{align*}
  We remark that $h_2$ satisfies \eqref{eq:kolmogorov} with
  $S\equiv 0$ in $Q_{(K+1)r}$. We thus can apply
  Corollary~\ref{cor:deriv} and get
  \begin{align*}
    \|h_2 - \mathcal T_{0}[h_2]\|_{L^\infty(Q_r)}
    &\lesssim r^2 \left[\frac{r}{((K+1)r)^3}
      +
      \frac{r^2}{((K+1)r)^4}+\frac{r}{((K+1)r)^3}+\frac{r}{((K+1)r)^3}\right]
      \|h_2\|_{L^\infty(Q_{(K+1)r})} \\
    & \lesssim (K+1)^{-3} \|h_2\|_{L^\infty(Q_{(K+1)r})}.
  \end{align*}
  Since $g-\mathcal T_{0} g = h_1 +h_2$, we can estimate
  $\|h_2\|_{L^\infty(Q_{(K+1)r})}$ as follows
  \begin{align*}
    \|h_2\|_{L^\infty(Q_{(K+1)r})}
    &\le \|h_1\|_{L^\infty(Q_{(K+1)r})} + \|g - \mathcal T_{0}[g] \|_{L^\infty(Q_{(K+1)r})} \\
    & \lesssim (K+1)^{2+\alpha} r^{2+\alpha} \left( [S]_{\cC_\ell^{\alpha}(Q_{(K+1)r})}
      + [g]_{\cC_\ell^{2+\alpha}(Q_{(K+1)r})} \right) 
  \end{align*}
  (we used \eqref{eq:g}) and get
  \begin{equation}\label{eq:h}
    \|h_2 -\mathcal T_{0}[h_2]\|_{L^\infty(Q_r)} \le C
    \frac{r^{2+\alpha}}{(K+1)^{1-\alpha}}
    \left([S]_{\cC_\ell^{\alpha}(Q_{(K+1)r})} + [g]_{\cC_\ell^{2+\alpha}(Q_{(K+1)r})}\right).
  \end{equation}
  Combining \eqref{eq:start}, \eqref{eq:g} and \eqref{eq:h}, we get 
  \begin{align*}
    \| g - \mathcal
    T_{0}[g] &- \mathcal T_{0} [h]_2 \|_{L^\infty(Q_r)} \\
           & \le \| h_2 - \mathcal T_{0} h_2 \|_{L^\infty(Q_r)} + \|
             h_1\|_{L^\infty(Q_r)} \\
           & \le C (K+1)^{2+\alpha} r^{2+\alpha}
             [S]_{\cC_\ell^{\alpha}(Q_{(K+1)r})}
             + C \frac{r^{2+\alpha}}{(K+1)^{1-\alpha}}
             [S]_{\cC_\ell^{\alpha}(Q_{(K+1)r})} +
             C \frac{r^{2+\alpha}}{(K+1)^{1-\alpha}}
             [g]_{\cC_\ell^{2+\alpha}(Q_{(K+1)r})} \\
  \end{align*}
  and by setting $K$ large enough so that
  $C (K+1)^{-(1-\alpha)} \le \frac12$, it results into
  \[
    [g]_{\cC_\ell^{2+\alpha}} \le \sup_{z_0,r>0} \frac{ \| g -
      \mathcal T_{z_0}[g] - \mathcal T_{z_0}[h_2]
      \|_{L^\infty(Q_r)}}{r^{2+\alpha}} \le C_K
    [S]_{\cC_\ell^{\alpha}}
  \]
  where we have used that
  $\mathcal T_{z_0}[g] + \mathcal T_{z_0}[h_2]$ is a polynomial of
  kinetic order less or equal to $2$, which concludes the proof. Note
  the \emph{corrector} $\mathcal T_{z_0}[h_2]$ to the Taylor
  polymomial $\mathcal T_{z_0}[g]$ used in this argument.
\end{proof}

The general constant coefficients case is reached through a change of
variables.
\begin{corollary}[Schauder estimates for
  constant diffusion coefficients]\label{cor:constant}
  Let $\alpha \in (0,1)$ and $S \in \cC_\ell^\alpha$ and let
  $A:=(a^{i,j})$ be a constant real $d \times d$ matrix that satisfies
  \eqref{eq:elliptic}. Then for all solution $g \in \cC_\ell^{2+\alpha}$ to
  \[
    (\partial_t + v \cdot \nabla_x) g = \sum_{1 \le i, j \le d}
    a^{i,j} \partial^2_{v_i v_j} g \quad \text{ in } \R \times \R^d
    \times \R^d
  \]
  we have,
  \[
    [g]_{\cC_\ell^{2+\alpha}} \le C [S]_{\cC_\ell^{\alpha}}
  \]
  where the constant $C$ depends on $d$, $\alpha$, the norm of the
  (constant) matrix $A$ and $\lambda$ in \eqref{eq:elliptic}.
\end{corollary}
\begin{remark}
This is the counterpart to \cite[Theorem~8.9.1,~p.~127]{krylov1996}.
\end{remark}
\begin{proof}
  Consider the change of variables
  \[
    \bar g(t,x,v) := g(t,B x,B v)
  \]
  with $B^2=A^{-1}$. Then we have for
  $(t,x,v) \in \R \times \R^d \times \R^d$,
  \begin{align*}
    (\partial_t + v \cdot \nabla_x) \bar g (t,x,v) - \Delta_v \bar g(t,x,v)
    &= (\partial_t + v \cdot \nabla_x)  g (t,B x,B v) 
    - \sum_{1 \le i,j \le d} a^{i,j} \partial^2_{v_i v_j} g(t,B x, B v) \\
    &=   S(t,Bx,Bv) =: \bar S (t,x,v)
  \end{align*}
  and the result follows from Theorem~\ref{thm:constant}.
\end{proof}

\subsubsection{Proof of Theorem~\ref{thm:variable}}

\begin{proof}[Proof of Theorem~\ref{thm:variable}]
  It is enough to treat the case where $b \equiv 0$ and $c \equiv 0$,
  the general case is then treated by interpolation (using
  Lemma~\ref{l:interpol}). We consider a constant $\gamma>0$ which
  will be fixed later, and pick
  $z_0,z_1 \in \R \times \R^d \times \R^d$ and $r >0$ such that
  $z_1 \in Q_r(z_2)$ and
  \[
    [g]_{\cC_\ell^{2+\alpha}} \le 2 \frac{\left| g (z_1) -
      \mathcal{T}_{z_0}[g](z_1)\right|}{r^{2+\alpha}}.
  \]
  \smallskip
  
  \noindent {\bf Case~1: $r \ge \gamma$.} We compute then
  \begin{align*}
    [g]_{\cC_\ell^{2+ \alpha}}
    & \le 2   \left( 2 \gamma^{-2-\alpha} \|g\|_{L^\infty}
      + \| (\partial_t + v \cdot \nabla_x) g\|_{L^\infty} \gamma^{-\alpha}
      + \|\nabla_v g \|_{L^\infty} \gamma^{-1-\alpha}
      + \frac12 \| D^2_v g \|_{L^\infty} \gamma^{-\alpha} \right) \\
    & \le  \frac14 [g]_{\cC_\ell^{2+\alpha}} + C_1(\gamma) \|g\|_{L^\infty}
  \end{align*}
  (using again Lemma~\ref{l:interpol}) with $C_1(\gamma) >0$ depending
  on $\gamma$ and $d$.

  \smallskip
  \noindent
  {\bf Case~2: $r \le \gamma$.} Then $z_1 \in Q_\gamma (z_0)$ and we
  consider a $C^\infty$ cut-off function $0 \le \xi \le 1$ that is
  equal to $1$ on $Q_\gamma(z_0)$ and equal to zero outside
  $Q_{2\gamma}(z_0)$.   
  In particular, $\xi (z_1)=\xi(z_0)=1$.  We now use
  Corollary~\ref{cor:constant} to get
  \begin{align*}
    [g]_{\cC_\ell^{2+ \alpha}}
    & \le 2  [g \xi]_{\cC_\ell^{2+\alpha}} 
     \le C \left[  (\partial_t  + v\cdot \nabla_x) (g\xi)
      - \sum_{i,j} a^{i,j}(z_1)
      \partial_{v_iv_j} ^2 (g\xi)\right]_{\cC_\ell^{\alpha}} \\
    & \le C \left[  (\partial_t  + v\cdot \nabla_x) (g\xi)
      - \sum_{i,j} a^{i,j}(\cdot)
      \partial_{v_iv_j} ^2 (g\xi) \right]_{\cC_\ell^{\alpha}}
      + C \left[\sum_{i,j} \left(a^{i,j}(\cdot) - a^{i,j}(z_1)\right)
      \partial_{v_iv_j} ^2 (g\xi) \right]_{\cC_\ell^{\alpha}(Q_{2 \gamma}(z_2))}
  \end{align*}
  where we have added the restriction to the support of $\xi$ in the
  last term. We estimate successively the two terms of the right hand
  side. On the one hand, recalling that
  $(\partial_t + v\cdot \nabla_x) g = \sum_{i,j} a^{i,j}\partial_{v_iv_j} ^2 g+S$,
  \begin{align*}
    \left[  (\partial_t  + v\cdot \nabla_x) (g\xi) -\sum_{i,j} a^{i,j}
    \partial^2_{v_i v_j}(g\xi) \right]_{\cC_\ell^{\alpha}}  
    & \le  \left[  (\partial_t  + v\cdot \nabla_x) g - \sum_{i,j} a^{i,j}
      \partial_{v_i v_j} ^2g \right]_{\cC_\ell^{\alpha}} \\
    & \qquad + \left[  g \left\{ (\partial_t  + v\cdot \nabla_x) (\xi)
      - \sum_{i,j} a^{i,j}
      \partial_{v_i v_j} ^2\xi \right\} \right]_{\cC_\ell^{\alpha}}
      + \left[\nabla_v g \cdot \nabla_v \xi\right]_{\cC_\ell^{\alpha}} \\
    & \lesssim_\gamma  [S]_{\cC_\ell^{\alpha}}
      + \|g\|_{\cC_\ell^{\alpha}} + \| \nabla_v g \|_{\cC_\ell^{\alpha}} \\
    & \le \frac14 [g]_{\cC_\ell^{2+\alpha}}
      + C_2(\gamma) \left([S]_{\cC_\ell^{\alpha}}+\|g\|_{L^\infty}\right)
  \end{align*}
  (using again Lemma~\ref{l:interpol}) for a constant $C_2(\gamma)>0$
  depending on $d$, $\|A\|_{\cC_\ell^{\alpha}}$, $\alpha$.  On the
  other hand,
  \begin{align*}
    \left[\sum_{i,j} \left(a^{i,j}(\cdot) - a^{i,j}(z_1)\right)
    \partial_{v_iv_j} ^2(g\xi) \right]_{\cC_\ell^{\alpha}(Q_{2\gamma} (z_2))} 
    & \lesssim \gamma^{\alpha} [D^2_v g]_{\cC_\ell^{\alpha}}
      + \|D^2_v g \|_{L^\infty} \\ 
    & \le C_3 \gamma^\alpha [g]_{\cC_\ell^{2+\alpha}} + C_4(\gamma) \|g\|_{L^\infty}
  \end{align*}
  using again Lemma~\ref{l:interpol}, for some constants $C_3>0$ and
  $C_4(\gamma)$ depending $\gamma$. Combining the last three
  estimates yields finally in the case $r \le \gamma$:
  \[
    [g]_{\cC_\ell^{2+\alpha}} \le \left( C_3 \gamma^\alpha + \frac14
    \right) 
    [g]_{\cC_\ell^{2+\alpha}} + \left( C_2(\gamma) + C_4(\gamma)
    \right) \left([S]_{\cC_\ell^\alpha} + \|g\|_{L^\infty}\right).
  \]
  
  We now pick $\gamma$ such that
  $C \gamma^\alpha + \frac14 \le \frac12$ and we thus get in both
  cases ($r \ge \gamma$ and $r \le \gamma$) that
  \[
    [g]_{\cC_\ell^{2+\alpha}} \le \frac12 [g]_{\cC_\ell^{2+\alpha}} + C_5(\gamma)
    \left([S]_{\cC_\ell^\alpha} + \|g\|_{L^\infty}\right)
  \]
  for some constant $C_5(\gamma)>0$, which concludes the proof
  of~\eqref{eq:schauder-main} thanks to~\eqref{eq:interp4-bis} and
  \eqref{eq:interp5-bis}.
\end{proof}

\subsection{Localization of the Schauder estimate}
\label{sec:local}

\begin{theorem}[Local Schauder estimate]\label{thm:variable-loc}
  Given $\alpha \in (0,1)$ and $a^{i,j},b^i,c, S \in\cC_\ell^{\alpha}$
  satisfying \eqref{eq:elliptic} for some constant $\lambda >0$, any
  solution $g \in \cC_\ell^{2+\alpha}$ to \eqref{eq:linear} satisfies
  for all $z_0 \in \R \times \R^d \times \R^d$
  \begin{align*}
    \| g\|_{\cC_\ell^{2+\alpha}(Q_1(z_0))}
    \le C \left([ S ]_{\cC_\ell^{\alpha}(Q_2(z_0))} + \| g
    \|_{L^\infty(Q_2(z_0))} \right)
  \end{align*}
  where the constant $C$ depends on $d$, $\lambda$ and $\alpha$ and
  $\|a^{i,j}\|_{\cC_\ell^{\alpha}}$, $\|b^i\|_{\cC_\ell^{\alpha}}$,
  $\|c\|_{\cC_\ell^{\alpha}}$ for $i,j=1,\dots,d$.
\end{theorem}
\begin{proof}
  We use the strategy of \cite[Theorem~8.11.1]{krylov1996}. Consider
  $z_0=0$ without loss of generality and define
  $R_n := \sum_{j =0}^n 2^{-j}$ for $n \ge 0$. Define a cutoff
  function $\zeta_n$ that is smooth, one on $Q_{R_n}$ and zero outside
  $Q_{R_{n+1}}$. It satisfies the controls
  \[
    \| \zeta_n \|_{\cC_\ell^{\alpha}}, \| v \cdot \nabla_x \zeta_n
  \|_{\cC_\ell^{\alpha}}, \| \nabla_v \zeta_n \|_{\cC_\ell^{\alpha}}, \|
  \nabla_v ^2 \zeta_n \|_{\cC_\ell^{\alpha}} \lesssim \rho^{-n} \quad
  \mbox{ with  } \quad \rho :=2^{2+\alpha}.
\]

Then apply the non-localized estimate of Theorem~\ref{thm:variable} to
$\zeta_n g$. In order to do so, it is convenient to consider the
differential operator
$\mathcal{L} = (\partial_t + v \cdot \nabla_x) - \sum_{i,j} a^{i,j}
\partial^2_{v_i v_j} - \sum_i b^i \partial_{v_i}$, recall
$\mathcal{L} g =S$ and write
  \begin{align*}
    A_n := [ g ]_{\cC_\ell^{2+\alpha}(Q_{R_n})}   \le
    & [ \zeta_n g]_{\cC_\ell^{2+\alpha}}
      \lesssim  [ \mathcal L (\zeta_n g)  ]_{\cC_\ell^{\alpha}}
      + \| \zeta_n g \|_{L^\infty} \\
    \lesssim
    &  \left[ \zeta_n \mathcal L g + g \mathcal{L} \zeta_n
      + (A \nabla_v g) \cdot \nabla_v \zeta_n
      \right]_{\cC_\ell^{\alpha}}
      + \| \zeta_n g \|_{L^\infty} \\
    \lesssim
    & [S]_{\cC_\ell^{\alpha}(Q_2)} + \rho^{-n} 
      \left( \| g\|_{\cC_\ell^{\alpha}(Q_{R_{n+1}})}
      + \| \nabla_v g\|_{\cC_\ell^{\alpha}(Q_{R_{n+1}})}\right)
      +\|g\|_{L^\infty(Q_{R_{n+1}})}  \\
    \intertext{ and use the interpolation inequalities from
    Lemma~\ref{l:interpol} to get for all $\eps_n >0$,}
    \lesssim
    &  [S]_{\cC_\ell^{\alpha}(Q_2)} + \rho^{-n} \left( \eps_n A_{n+1}
      + \eps_n^{-\beta} \rho^{-n} \|g\|_{L^\infty(Q_2)}\right)
      + \|g\|_{L^\infty(Q_{R_{n+1}})} 
  \end{align*}
  for some $\beta >0$ (note the additional factor
    $\rho^{-n}$ due to the dependency of the constant $\tilde{C}$ in
    the difference of radii in Lemma~\ref{l:interpol}). Choosing next
  $\eps_n := \eps_0 \rho^n$ for $\eps_0 \in (0,1)$ small enough yields
  \begin{align*}
    A_n \lesssim [S]_{\cC_\ell^{\alpha}(Q_2)} + \eps_0 A_{n+1}
    + \eps_0 ^{-\beta} \rho^{-(\beta+2) n} \| g \|_{L^\infty(Q_2)}
  \end{align*}
  Consider then the geometric sum $\sum_{n \ge 0} \eps_0 ^n A_n$, and  calculate
  \[
    \sum_{n \ge 0} \eps_0^n  A_n    \lesssim \left(\sum_{n \ge 0}
      \eps_0^n \right)
    [S]_{\cC_\ell^{\alpha}( Q_{2})} + \sum_{n \ge 0}
    \eps_0^{n+1} A_{n+1} + \eps_0 ^{-\beta}
    \sum_{n \ge 0} \left(
      \frac{\eps_0}{\rho^{\beta+2}} \right) ^n
    \| g \|_{L^\infty(Q_2)}. 
  \]
  Assuming $\eps_0 < \rho^{\beta+2}<1$ and cancelling terms gives
  finally:
  \[
    A_0 \lesssim  [S]_{\cC_\ell^{\alpha}(Q_{2})} +  \| g \|_{L^\infty(Q_2)}
  \]
  which concludes the proof. 
\end{proof}

\section{Global existence for the toy model}
\label{sec:toy}

This section is devoted to the proof of
Theorem~\ref{theo:main-holder}. Equation~\eqref{eq:main} is rewritten
with the unknown function $g := f \mu^{-\frac12}$, where $\mu$ denotes
the Gaussian $(2\pi)^{-d/2} e^{-|v|^2/2}$. This new function satisfies
\begin{equation}\label{eq:toyagain}
  \partial_t g + v \cdot \nabla_x g = \fR[g] U[g]
\end{equation}
\[
  \mbox{ with } \quad 
  \fR[g] := \int_v g \sqrt\mu \dd v \quad \mbox{ and } \quad U[g] :=
  \Delta_v g + \left(\frac{d}2 - \frac{|v|^2}{4}\right) g =
  \frac1{\sqrt{\mu}} \nabla_v \left(\mu \nabla_v
    \left(\frac{1}{\sqrt{\mu}} g\right)\right).
\]
After this rescaling the operator $U$ is symmetric in
$L^2(\dd x \dd v)$, without weight.
In contrast with \eqref{eq:main}, this operator has no first order
term in the velocity variable. But a (simpler) difficulty is created
with the appearance of the unbounded zero order term
$(\frac{d}2 - \frac{|v|^2}4)g$. We overcome it using that $g$ stays in
between two Maxwellians (see Lemma~\ref{l:gaussian} below) and that,
after rescaling, the H\"older estimate from \cite{wz09,gimv} in
Theorem~\ref{t:gimv} encodes decay in the $v$ variable
(Proposition~\ref{prop:local}).

\subsection{Gaussian bounds}

We first explain how to propagate in time the Gaussian bounds
satisfied by the initial data $g_{in} (x,v) = g(0,x,v)$.
\begin{lemma}[Gaussian bounds]\label{l:gaussian}
  Consider a strong solution $g$ to \eqref{eq:toyagain} in
  $L^\infty_t([0,T],H^2(\T^d\times\R^d))$ such that
  $C_1 \sqrt{\mu} \le g(0,x,v) \le C_2 \sqrt{\mu}$ in
  $\T^d \times \R^d$, then for almost all $t \in [0,T]$, we have
  \[
    \forall \, (x,v) \in \T^d \times \R^d, \quad C_1 \sqrt{\mu(v)} \le
    g(t,x,v) \le C_2 \sqrt{\mu(v)}.
  \]
\end{lemma}
\begin{remark}
  The notion of strong solutions used in the latter lemma and the rest
  of this section is standard: all the terms in the equation are
  defined almost everywhere and the equation is satisfied almost
  everywhere and the solution is continuous in time with value in
  $L^2$. We could have considered weaker notions of solutions but it
  was unnecessary since we are interested in constructing global
  \emph{strong} solutions.
\end{remark}
\begin{proof}
  For $a \in \R$, we write $a_+ = \max (0,a)$ and $a_- =
  \max(0,-a)$. It is enough to prove that, for any $C>0$,
  \[
    \int_{\T^d \times \R^d} (g (t,x,v) - C \sqrt{\mu}
    (v))^2_\pm \dd x \dd v
  \]
  decreases along time. It follows from the fact that
  $\tilde g = g - C\sqrt\mu$ satisfies
  \[
    \partial_t \tilde g + v \cdot \nabla_x \tilde g = \fR[g] U[\tilde g]
  \]
  and, multiplying the previous equation by
  $(\tilde g)_+ = \max (0,\tilde g)$ or
    $(\tilde g)_- = \max(0,-\tilde g)$, and integrating with respect
  to $x$ and $v$ yields for almost all $t \in [0,T]$,
  \[
    \int_{\T^d \times \R^d} (g (t,x,v) - C \sqrt{\mu} (v))^2_\pm \dd
    x \dd v  \le \int_{\T^d \times \R^d} (g_{in} (x,v) - C \sqrt{\mu} (v))^2_\pm \dd
    x \dd v .
  \]
  In particular, for almost all $t \in [0,T]$,
  \begin{align*}
     \int_{\T^d \times \R^d} (g (t,x,v) - C_2 \sqrt{\mu} (v))^2_+ \dd
    x \dd v  \le \int_{\T^d \times \R^d} (g_{in} (x,v) - C_2 \sqrt{\mu} (v))^2_+ \dd
    x \dd v = 0 \\
     \int_{\T^d \times \R^d} (g (t,x,v) - C_1 \sqrt{\mu} (v))^2_- \dd
    x \dd v  \le \int_{\T^d \times \R^d} (g_{in} (x,v) - C_1 \sqrt{\mu} (v))^2_- \dd
    x \dd v = 0.
  \end{align*}
  We conclude that for almost all $t \in [0,T]$,
  $C_1 \sqrt\mu (v) \le g (t,x,v) \le C_2 \sqrt\mu (v)$ in
  $\T^d \times \R^d$.
\end{proof}

\subsection{The local H\"older estimate for kinetic Fokker-Planck
  equations}
\label{ss:holder}

We recall in this subsection the main result from \cite{wz09,gimv} as
stated in \cite{gimv} and used in the proof of
Theorem~\ref{theo:main-holder}. Consider the equation
\begin{equation}\label{e:lin-div}
  \partial_t g + v \cdot \nabla_x g =
  \nabla_v( A \nabla_v g) + S_0 \quad \text{ in } \quad Q_{2r}(z_0)
\end{equation}
with $A=A(t,x,v)$ real symmetric matrix whose eigenvalues lie in
$[\lambda,\Lambda]$ for $0 < \lambda \le \Lambda$.

\begin{theorem}[Local H\"older estimate]\label{t:gimv}
  Given a cylinder $Q_{2r}(z_0) \subset \R \times \R^d \times \R^d$
  and a function $S_0$ essentially bounded in $Q_{2r}(z_0)$, any weak
  solution $g$ of \eqref{e:lin-div} in $Q_{2r}(z_0)$ such that
  $g \in L^\infty_t([0,T], L^2_{x,v} (\R^d \times \R^d)) \cap L^2_{t,x} ([0,T]
  \times \R^d, H^1_v (\R^d))$ satisfies
  \[
    \| g\|_{\cC_\ell^{{\alpha_0}}(Q_{\frac{3r}2}(z_0))} \le C
    r^{-{\alpha_0}} \| g \|_{L^2(Q_{2r}(z_0))} + C r^{-2-{\alpha_0}}
    \|S_0\|_{L^\infty(Q_{2r}(z_0))}
  \]
  for some constants $\alpha_0 \in (0,1)$ and $C>0$ only depending on
  $d$ and $\lambda,\Lambda$.
\end{theorem}

\subsection{The Schauder estimate for the toy model}
\label{sec:scha-estim-toy}

The Schauder estimate follows from (1) the H\"older regularity
established in \cite{wz09,gimv}, (2) the Schauder estimate from
Section~\ref{sec:schauder} and (3) the Gaussian bounds from the
previous subsection. Since the H\"older regularity only holds for
positive times, we consider some time interval $[\tau,T]$ for some
arbitrary but fixed $\tau>0$.

\begin{proposition}[Higher order H\"older  Estimates]\label{prop:local}
  There exists ${\alpha_0} \in (0,1)$, only depending on the dimension
  $d$ and $C_1,C_2$, and $C>0$ only depending on
  $C_1,C_2,d, \tau,\delta$ such that for all solution
  $g \in \cC_\ell^{2+{\alpha_0}}$ to \eqref{eq:toyagain} that
  satisfies $C_1 \sqrt \mu \le g \le C_2 \sqrt \mu$, for all
  $\tau \in (0,T)$ and $\delta \in (0,1)$, we have
  \begin{equation}
    \label{eq:decay}
    \| g \|_{\cC_\ell^{2+{\alpha_0}}(Q_r(z_0))} \le C \mu^\delta (v_0).
 \end{equation}
  for any 
  $Q_{2r} (z_0) \subset [\tau, T] \times \T^d \times \R^d$. In
  particular,
 \begin{equation}
   \label{e:decay-grad}
   \| \mu^{-\delta} \nabla_v g \|_{L^\infty([\tau,T] \times \R^d \times \R^d)} \le C.
 \end{equation}
\end{proposition}
\begin{proof}
  We apply Theorem~\ref{t:gimv} with $\lambda = C_1$, $\Lambda = C_2$,
  $A = \fR[g] \mbox{Id}$ and $S_0 = \fR[g](\frac{d}2-\frac{|v|^2}4)g$ and get
  \begin{align*}
    \| g\|_{\cC_\ell^{{\alpha_0}}(Q_{\frac{3r}2}(z_0))}  
    \le C r^{-{\alpha_0}} \| g \|_{L^2(Q_{2r}(z_0))}  + C
    r^{-2-{\alpha_0}}
    \left\|\left(\frac{d}2 - \frac{|v|^2}4\right)g \right\|_{L^\infty(Q_{2r}(z_0))}
  \end{align*}
  for some $\alpha_0 \in (0,1)$. Apply now Lemma~\ref{l:gaussian} to
  get for an arbitrary $\delta \in (0,1)$
  \begin{align*}
    \| g\|_{\cC_\ell^{{\alpha_0}}(Q_{\frac{3r}2}(z_0))} \le C_\delta \mu^{\delta}(v_0)
  \end{align*}
  for some constant $C_\delta$ depending only on
  $C_1, C_2, d,\tau, \delta$. Note that the role of the time $\tau>0$
  is to ensure that the constant $C_\delta$ only depends on $\tau$ and
  not on $r$: it gives some ``room'' around a point
  $z_0 \in [\tau,T] \times \T^d \times \R^d$.
  
  This implies in turn that for all $(t,x) \in [\tau,T] \times \T^d$,
  \begin{align*}
    |\fR[g](t,x) - \fR[g](s,y)|
    & \le \int_{\R^d} | g (t,x,v) - g (s,y,v)|\mu^{\frac12} (v) \dd v \\
    & \le C_\delta  \int_{\R^d} \left(|t-s|^{\frac{\alpha_0}2}
      + |x-y - (t-s)v|^{\frac{\alpha_0}3} \right) \mu^{\frac12+\delta} (v) \dd v \\
    & \lesssim C_\delta \left(|t-s|^{\frac{\alpha_0}2}
      + |x-y|^{\frac{\alpha_0}3} + |t-s|^{\frac{\alpha_0}3}\right).
  \end{align*}

  This also implies for any $\delta' \in (0,\delta)$
  \[
    \left\|\left(\frac{d}2-\frac{|v|^2}4\right) g\right\|
    _{\cC_\ell^{{\alpha_0}}(Q_{\frac{3r}2}(z_0))}  \le
    C_\delta \mu^{\delta'}(v_0) .
  \]

  This ensures the H\"older regularity of the coefficients and source
  term, and we are thus in a position to apply
  Theorem~\ref{thm:variable-loc} in the cylinders $Q_r (z_0)$
  and deduce \eqref{eq:decay}. To get \eqref{e:decay-grad} from
  \eqref{eq:decay} and the decay of Lemma~\ref{l:gaussian}, apply
  Lemma~\ref{l:interpol} on a cylinder $Q$ to obtain
  $\| \nabla_v g \|_{L^\infty(Q)} \le C \|g\|_{L^\infty(Q)} + C [
  g]_{\cC_\ell^{2+\alpha_0}(Q)}$ and argue as before. This
  achieves the proof of the lemma.
\end{proof}

\subsection{Standard interpolation product inequality}
\label{sec:stand-interp-prod}

We recall and prove an interpolation inequality tailored to our needs;
it is folklore knowledge.
\begin{lemma}\label{lem:product}
  Consider $k \ge d/2$, two functions $g_1, g_2 \in H^k(\T^d)$ and
  ${\bar m}, m \in \N^d$ such that  $|{\bar m}| + |m| = k$
  then
  \begin{equation*}
    \left\| \partial^{\bar m}_x g_1 \partial_x ^m g_2 \right\|_{L^2(\T^d)}
      \lesssim_k \| g_1 \|_{L^\infty(\T^d)} \| g_2 \|_{H^k(\T^d)} +
      \| g_1 \|_{H^k(\T^d)} \| g_2 \|_{L^\infty(\T^d)}.
  \end{equation*}
  Moreover whenever ${\bar m} \not =0$, for any $\eps>0$ there
  is $C_\eps>0$ s.t.
  \begin{equation}\label{eq:producteps}
    \left\| \partial^{\bar m}_x g_1 \partial_x ^m g_2 \right\|_{L^2(\T^d)}
      \le \eps \| g_1 \|_{L^\infty(\T^d)} \| g_2 \|_{H^k(\T^d)} +
      C_\eps  \| g_1 \|_{H^k(\T^d)} \| g_2 \|_{L^\infty(\T^d)}.
  \end{equation}
\end{lemma}
\begin{proof}
  The first inequality is clear when $\bar m=0$. Use
    otherwise the \emph{Nash inequality}: given $h \in H^k(\T^d)$
  with $k >d/2$ and any ${\bar m} \in \N^d$, with
  $0 < |{\bar m}| \le k$ then for
  $p := \frac{2k}{|{\bar m}|} \in [1,+\infty]$ one has
  \begin{equation*}
    \left\| \partial^{\bar m}_x h \right\|_{L^p(\T^d)} \lesssim \| h
    \|_{L^\infty(\T^d)} ^{1-\frac{|{\bar m}|}k} \| h \|_{H^k(\T^d)}
    ^{\frac{|{\bar m}|}k}. 
  \end{equation*}

  Apply successively the H\"older inequality with
  $p := \frac{2k}{|{\bar m}|}$ and $q := \frac{2k}{|m|}$, and Nash
  inequality above:
  \begin{align*}
    \left\| \partial^{\bar m}_x g_1 \partial_x ^m g_2 \right\|_{L^2(\T^d)}
    & \le \left\| \partial^{\bar m}_x g_1 \right\|_{L^p(\T^d)}  \left\|
      \partial^{\bar m}_x g_2 \right\|_{L^q(\T^d)} 
     \lesssim \| g_1\|_{L^\infty(\T^d)} ^{1-\frac{|{\bar m}|}{k}}
      \| g_1 \|_{H^k(\T^d)}^{\frac{|{\bar m}|}k} \| g_2
    \|_{L^\infty(\T^d)} ^{1-\frac{|m|}k} \| g_2 \|_{H^k(\T^d)}
      ^{\frac{|m|}k} \\
    & \lesssim \left( \| g_1\|_{L^\infty(\T^d)} \| g_2
      \|_{H^k(\T^d)}\right)^{1-\frac{|{\bar m}|}k}
      \left( \| g_1 \|_{H^k(\T^d)} \| g_2\|_{L^\infty(\T^d)}
      \right)^{\frac{|{\bar m}|}k}.
  \end{align*}
  Use finally
  $a^{1-\lambda} b^\lambda \le (1-\lambda)a + \lambda b$ to deduce the
  two claimed inequalities.
\end{proof}

\subsection{Local well-posedness in Sobolev spaces}
\label{sec:local-wellp-sobol}

We prove that equation~\eqref{eq:main} is locally well-posed in
$H^k (\T^d \times \R^d)$. It is convenient to work with the
following norm,
\begin{align*}
  \| g \|_{H^{k,\bar k}_{x,v}(\T^d \times \R^d)} ^2 := \| g \|_{L^2(\T^d
  \times \R^d)} ^2 + \sum_{i=1,\dots,d} \left\| \partial_{x_i} ^k
  h\right\|_{L^2(\T^d \times \R^d)}^2
  + \sum_{i=1,\dots,d} \left\| \partial_{v_i} ^{\bar k}
  h\right\|_{L^2(\T^d \times \R^d)}^2. 
\end{align*}
Note that this norm agrees with the usual $H^k$ norm in $x$ and $v$
when $k = \bar k$: this is clear from the Fourier-transformed
definition of these norms. However for $k \not = \bar k$ the space
$H^{k,\bar k}_{x,v}$ includes but differs from
$H^{\min(k,\bar k)}_{x,v}$. We use this intermediate definition of
norm solely to shorten some calculations.
\begin{theorem}[Local well-posedness in $H^k$]\label{t:loc-exis}
  Consider $g_{in} \in H^{k,\bar k}_{x,v}(\T^d \times \R^d)$ with
  $k,{\bar k}$ non-negative integers s.t. $2 \le{\bar k} \le k$ and
  $k > d/2$, and $C_1 \sqrt{\mu} \le g_{in} \le C_2 \sqrt{\mu}$ for
  $0 < C_1 < C_2$.  Then there is $T>0$ depending only on $C_1, C_2$
  and $\|g_{in}\|_{H^{k,\bar k}_{x,v}(\T^d \times \R^d)}$ such that
  there exists a unique strong solution
  $g \in C^1([0,T],H^{k,\bar k}_{x,v}(\T^d \times \R^d))$ to
  \eqref{eq:toyagain} with initial data $g_{in}$, which furthermore
  satisfies $C_1 \sqrt{\mu} \le g (t,\cdot,\cdot) \le C_2 \sqrt{\mu}$
  for almost all $t \in [0,T]$.
\end{theorem}
\begin{proof}
  It is standard calculations that the two a priori
  estimates~\eqref{eq:gronwall} and~\eqref{eq:difference} below imply
  the existence of solutions constructed through the iterative scheme
  \begin{equation}
    \label{eq:iterated}
    \left\{
      \begin{array}{l}
     (\partial_t  + v \cdot \nabla_x ) g_{n+1} = \mathcal{R} [g_n]
        U[g_{n+1}],\\[3mm]
        g_{n+1}(t=0,\cdot,\cdot) = g_{in} \\[3mm]
     \mbox{Initialization: }
        g_0(t,x,v) := g_{in}(x,v) \quad \forall \, t \ge 0.
      \end{array}
      \right.
    \end{equation}
    
    We now focus on establishing the key a priori estimate. Consider a
    solution
    $g_{n+1} \in C^1([0,T],H^{k,\bar k}_{x,v}(\T^d \times \R^d))$
    to~\eqref{eq:iterated} and compute successively:

\begin{equation}\label{eq:evol-2}
L^2 \mbox{ estimate: } \qquad \dt \frac{1}2  \int_{\T^d \times \R^d} | g_{n+1}|^2\dd x \dd v \le - C_1 
\int_{\T^d \times \R^d} |h_{n+1}|^2 \dd x \dd v
\end{equation}
where we denote $h_{n+1}:=\mu^{1/2} \nabla_v (\mu^{-1/2}
g_{n+1})$. Regarding the $v$-derivatives, for any integer
${\bar k} \ge 1$,
\begin{align*}
  \dt \frac12 \int_{\T^d \times \R^d} |\partial^{{\bar k}}_{v_i} g_{n+1}|^2 \dd x \dd
  v = &-  {\bar k} \int_{\T^d \times \R^d} (\partial^{{\bar k}-1} _{v_i}    \partial_{x_i} g_{n+1} ) \partial^{{\bar k}}_{v_i} g_{n+1} \dd x \dd v \\
      &- \int_{\T^d \times \R^d} \fR[g_n] \left| \nabla_v \left(\frac{\partial^{{\bar k}} _{v_i} g_{n+1}}{\sqrt \mu} \right) \right|^2 \mu  \dd x \dd v\\ 
    & + \frac14  \binom{{\bar k}}1 \int_{\T^d \times \R^d} \fR[g_n] \left| \partial^{{\bar k} - 1} _{v_i} g_{n+1} \right|^2  \dd x \dd v \\
    & + \frac12  \binom{{\bar k}}2 \int_{\T^d \times \R^d} \fR[g_n] \left|     \partial^{{\bar k}-1} _{v_i} g_{n+1} \right|^2 \dd x \dd v.
\end{align*}
In the right hand side, the first term corresponds to the transport
$v \cdot \nabla_x$, the second one to the operator $U$ since
$\fR[g_n]$ does not depend on $v$, the third term appears when one
$v$-derivative applies to $|v|^2$ and the others apply to $g_{n+1}$ in
the product $|v|^2 g_{n+1}$ appearing in $U[g_{n+1}]$, the fourth term
appears after deriving $|v|^2$ twice. Notice that integrations by
parts are used either to further differentiate $|v|^2$ or to make
appear $|\partial_{v_i}^{{\bar k}-1}g_{n+1}|^2$.  Discarding the
negative term and using the fact that $\fR[g_n] \le C_2$ thanks to
Lemma~\ref{l:gaussian}, we get after summing over $i =1,\dots,d$ and
combining with equation~\eqref{eq:evol-2}
\begin{align}\label{eq:evol-v}
  \mbox{Estimate of $v$-derivatives: } \quad
  \dt \frac12 \|g_{n+1}(t) \|^2_{H^{0,\bar k}_{x,v} (\T^d \times \R^d)}
  & \lesssim_{{\bar k},C_2}
    \left\| g_{n+1} \right\|_{H^{\bar k,\bar k}_{x,v}(\T^d \times \R^d)} ^2.
\end{align}
Regarding the $x$-derivatives, we write the equation on $g_{n+1}$ as
\[
  (\partial_t + v \cdot \nabla_x) g_{n+1} = \fR [g_n] \mu^{-\frac12}
  \nabla_v \cdot \left(\mu^{\frac12} \nabla_v h_{n+1} \right) \quad
  \mbox{ with } \quad
  h_{n+1}:= \mu^{\frac12} \,\nabla_v \left(\mu^{-\frac12} g_{n+1}\right).
\]
Since $x$-derivatives commute with the operators $v\cdot
\nabla_x$ and $U$, we have for all $i=1,\dots,d$,
\begin{align*}
 \dt \frac12 \int_{\T^d \times \R^d} |\partial^{k}_{x_i} g_{n+1}|^2 \dd x \dd v = 
 & - \int_{\T^d \times \R^d} \fR[g_n] | \partial_{x_i}^k  h_{n+1}|^2  \dd x \dd v\\ 
    &  - \sum_{0 \le q < k} \binom{k}{q} \int_{\T^d \times \R^d}
      (\partial^{k - q} _{x_i} \fR[g_n])  \partial^{q} _{x_i}
      h_{n+1} \cdot \partial^{k}_{x_i} h_{n+1} \dd x \dd v.
\end{align*}
Use now $\fR[g_n] \ge C_1$ from Lemma~\ref{l:gaussian}:
\begin{align*}
  \dt \frac12 \int_{\T^d \times \R^d} |\partial^{k}_{x_i} g_{n+1}|^2
  \dd x \dd v \lesssim_k
  & - C_1 \int_{\T^d \times \R^d}  | \partial_{x_i}^{k}h_{n+1}|^2 \dd x \dd v \\
  & + \sum_{0 \le q < k}   \int_{\T^d \times \R^d} \left|\partial^{k - q} _{x_i}
    \fR[g_n]\right| \cdot |\partial_{x_i}^{q} h_{n+1}|
    \cdot |\partial_{x_i}^{k} h_{n+1}| \dd x \dd v.
\end{align*}
Observe that, given $0 \le q < k$, the index $p:=k-q \not =0$ and the
inequality~\eqref{eq:producteps} in Lemma~\ref{lem:product} can be
applied (we use again below the upper bound $\fR[g_n] \le C_2$):
\begin{align*}
  &  \int_{\T^d \times \R^d} \left|\partial^{k-q} _{x_i}
    \fR[g_n]\right|
    \cdot |\partial_{x_i}^{q} h_{n+1}| \cdot
    |\partial_{x_i}^{k} h_{n+1}| \dd x \dd v  \\
  & = \int_{\R^d} \left( \int_{\T^d} \left|\partial^{p} _{x_i}
    \fR[g_n]\right|
    \cdot |\partial_{x_i}^{q} h_{n+1}|
    \cdot |\partial_{x_i}^{k}    h_{n+1}| \dd x \right) \dd v \\
  & \le \int_{\R^d} \left( \left\| \partial^{p} _{x_i}    \fR[g_n](t,\cdot) \partial_{x_i}^{q} h_{n+1}(t,\cdot,v) \right\|_{L^2_x(\T^d)}
    \left\|\partial_{x_i}^{k} h_{n+1}(t,\cdot,v) \right\|_{L^2_x(\T^d)}    \right) \dd v \\
  &\le \eps \int_{\R^d} \left( \left\|    \fR[g_n](t,\cdot)\right\|_{L^\infty_x(\T^d)} \left\| h_{n+1}(t,\cdot,v)    \right\|_{H^k_x(\T^d)} ^2\right) \dd v \\
  &  + C_\eps \int_{\R^d} \left( \left\|    \fR[g_n](t,\cdot)\right\|_{H^k _x(\T^d)} \left\| h_{n+1}(t,\cdot,v) \right\|_{L^\infty_x(\T^d)} \left\| h_{n+1}(t,\cdot,v)    \right\|_{H^k_x(\T^d)} \right) \dd v\\
  &\le \eps \left( C_2 +    1\right)
    \left\| h_{n+1}(t,\cdot,\cdot) \right\|_{H^{k,0}_{x,v}(\T^d \times \R^d)} ^2 
    + C_\eps '\left\|  \fR[g_n](t,\cdot)\right\|_{H^k _x(\T^d)} ^2
    \int_{\R^d} \left\| h_{n+1}(t,\cdot,v) \right\|_{L^\infty_x(\T^d)} ^2 \dd v
\end{align*}
for any $\eps >0$ and some corresponding constant
$C_\eps,C_\eps'>0$. Use then
\[\left\| \fR[g_n](t,\cdot)\right\|_{H^k _x(\T^d)} \le \left\|
  g_n(t,\cdot,\cdot)\right\|_{H^{k,0}_{x,v}(\T^d\times \R^d)}\]
and equation~\eqref{eq:evol-2} to get
\begin{align*}
  \dt \frac12 \|g_{n+1}\|^2_{H^{k,0}_{x,v}(\T^d \times \R^d)}
  \lesssim_{k,C_2}
  & - C_1 \left\|
    h_{n+1}(t,\cdot,\cdot)\right\|_{H^{k,0}_{x,v}(\T^d \times
    \R^d)} ^2  
   + \eps \left\|
    h_{n+1}(t,\cdot,\cdot)\right\|_{H^{k,0}_{x,v}(\T^d \times
    \R^d)} ^2 \\
  & + C'_\eps \left( \int_{\R^d} \left\| h_{n+1}(t,\cdot,v)
    \right\|_{L^\infty_x(\T^d)} ^2 \dd v \right) \left\|
    g_n(t,\cdot,\cdot)\right\|_{H^{k,0}_{x,v}(\T^d \times \R^d)} ^2.
\end{align*}
Finally choose $\eps$ small enough (in terms of absolute constants,
independently of the solution) to get
\begin{equation}\label{eq:evol-x}
  \dt \frac12 \|g_{n+1}\|^2_{H^{k,0}_{x,v}(\T^d \times \R^d)}
  \lesssim_{k,C_1,C_2}  \left( \int_{\R^d} \left\|
      h_{n+1}(t,\cdot,v) \right\|_{L^\infty_x(\T^d)} ^2 \dd v \right)
  \left\| g_n(t,\cdot,\cdot) \right\|_{H^{k,0}_{x,v}(\T^d \times
      \R^d)} ^2.
\end{equation}

The combination of equations~\eqref{eq:evol-v} and~\eqref{eq:evol-x}
yields
\begin{align}
  \label{eq:evol-global}\nonumber
  \dt \frac12 \left\| g_{n+1}(t,\cdot,\cdot)\right\|_{H^{k,\bar
  k}_{x,v}(\T^d \times \R^d)} ^2
  & \lesssim_{k,C_1,C_2}
    \left\| g_{n+1} (t,\cdot,\cdot) \right\|_{H^{\bar k,\bar
    k}_{x,v}(\T^d \times \R^d)} ^2 \\ 
  & + \left( \int_{\R^d} \left\| h_{n+1}(t,\cdot,v)
    \right\|_{L^\infty_x(\T^d)} ^2 \dd v \right) \left\|
    g_n(t,\cdot,\cdot)  \right\|_{H^{k,0}_{x,v}(\T^d \times
    \R^d)} ^2.
\end{align}

Now observe that $h_{n+1} = \nabla_v g_{n+1} + \frac{v}{2} g_{n+1}$
and use $g_{n+1} \le C_2 \sqrt\mu$ and $k>d/2$ and ${\bar k} \ge 1$
and Sobolev embedding in $\T^d$ to get
\begin{align*}
   \int_{\R^d} \left\|  h_{n+1}(t,\cdot,v) \right\|_{L^\infty_x(\T^d)} ^2 \dd v 
 & \lesssim  C_2^2+  \int_{\R^d} \| \nabla_v g_{n+1} \|_{L^\infty_x(\T^d)} ^2 \dd v  \\
&  \lesssim C_2^2 + \| g_{n+1} \|_{H^{k,\bar k}_{x,v}(\T^d \times \R^d)}^2
\end{align*}
and conclude finally that
\begin{equation}
  \label{eq:gronwall}
  \dt \frac12 \left\| g_{n+1}(t,\cdot,\cdot)\right\|_{H^{k,\bar
        k}_{x,v}(\T^d \times \R^d)} ^2  
  \lesssim_{k,C_1,C_2}
  \left( \left\| g_n(t,\cdot,\cdot) \right\|_{H^{k,\bar
          k}_{x,v}(\T^d \times \R^d)} ^2 + 1 \right)
  \left\| g_{n+1}(t,\cdot,\cdot) \right\|_{H^{k,\bar
        k}_{x,v}(\T^d \times \R^d)} ^2
\end{equation}
which is the first main a priori estimate, that shows that the
$H^{k,\bar k}_{x,v}(\T^d\times \R^d)$ norm remains finite on a
short time interval (whose length depends on the size of the initial
data) thanks to Gronwall's lemma.  \medskip

Regarding uniqueness, consider the difference of two solutions
$g,{\bar g} \in H^{k,\bar k}_{x,v}(\T^d \times \R^d)$ that
satisfies
\begin{equation*}
  \partial_t(g-{\bar g}) + v \cdot \nabla_x (g-{\bar g}) =
  \fR[g-{\bar g}]U[g]+ \fR[{\bar g}] U[g-{\bar g}], 
\end{equation*}
and perform similar energy estimates to get
\begin{align}
  \label{eq:difference}\nonumber
  & \dt \frac12 \left\| (g-{\bar g})(t,\cdot,\cdot)
    \right\|_{H^k_{x}H^{\bar k}_v(\T^d \times \R^d)} ^2
    \lesssim_{k,C_1,C_2} \left\| (g-{\bar g})(t,\cdot,\cdot)
    \right\|_{H^{k,\bar k}_{x,v}(\T^d \times \R^d)} ^2 \\ 
  & + \left( \left\| g(t,\cdot,\cdot)
    \right\|_{H^{k,\bar k}_{x,v}(\T^d \times \R^d)} ^2
    + \left\| {\bar g}(t,\cdot,\cdot)
    \right\|_{H^{k,\bar k}_{x,v}(\T^d \times \R^d)} ^2 \right)
    \left\| (g-{\bar g})(t,\cdot,\cdot)
    \right\|_{H^{k,\bar k}_{x,v}(\T^d \times \R^d)}^2
\end{align}
which implies uniqueness in the space
$H^{k,\bar k}_{x,v}(\T^d \times \R^d)$ and concludes the proof. 
\end{proof}

\subsection{From local-in-time to global-in-time}
\label{sec:from-local-time}

To continue the solutions for all times it is enough to prove that the
$H^{k,\bar k}_{x,v}(\T^d \times \R^d)$ norm remains finite over
arbitrarily long times. Thanks to the local-in-time existence result
in Theorem~\ref{t:loc-exis}, we know that solutions exist at least on
a time interval $[0,2 \tau]$ for some $\tau>0$. We now show that the
Sobolev norm cannot explode after time $\tau$ by using
Proposition~\ref{prop:local}.

Consider again the a priori estimate~\eqref{eq:evol-global},
\begin{align*}
  \dt \frac12 \left\| g_{n+1}(t,\cdot,\cdot) \right\|_{H^{k,\bar
  k}_{x,v}(\T^d \times \R^d)} ^2  
  & \lesssim_{k,C_1,C_2} \left\| g_{n+1}(t,\cdot,\cdot)
    \right\|_{H^{\bar k,\bar k}_{x,v}(\T^d \times  \R^d)}^2 \\
  & + \left( \int_{\R^d} \left\| h_{n+1}(t,\cdot,v)
    \right\|_{L^\infty_x(\T^d)} ^2 \dd v \right)
     \left\| g_n(t,\cdot,\cdot) \right\|_{H^{k,0}_{x,v}(\T^d \times \R^d)}^2.
\end{align*}

Recall that $h_{n+1} = \nabla_v g_{n+1} + \frac{v}{2} g_{n+1}$
and use Lemma~\ref{l:gaussian}:
\begin{equation*}
  \left( \int_{\R^d} \left\| h_{n+1}(t,\cdot,v)
    \right\|_{L^\infty_x(\T^d)} ^2 \dd v \right)
 \lesssim \left( \int_{\R^d} \left\| \nabla_v g_{n+1}(t,\cdot,v)
   \right\|_{L^\infty_x(\T^d)} ^2 \dd v \right) + C_2.
\end{equation*}
Apply finally Proposition~\ref{prop:local}:
\begin{equation*}
  \left( \int_{\R^d} \left\|\nabla_v g_{n+1}(t,\cdot,v) \right\|_{L^\infty_x(\T^d)} ^2 \dd v
 \right) \lesssim \| \mu^{-\delta} \nabla_v g_{n+1} \|^2_{L^\infty([\tau,T]\times
   \T^d \times \R^d)} \lesssim  1. 
\end{equation*}

It shows for $t \ge \tau$, i.e. after some arbitrarily small time
$\tau>0$:
\begin{align*}
  \dt \frac12 \left\| g_{n+1}(t,\cdot,\cdot)  \right\|_{H^{k,\bar
  k}_{x,v}(\T^d \times \R^d)}^2
  \lesssim_{k,C_1,C_2,\tau} \left\| g_{n+1}(t,\cdot,\cdot)
  \right\|_{H^{k,\bar k}_{x,v}(\T^d \times \R^d)} ^2 +
  \left\| g_n(t,\cdot,\cdot)  \right\|_{H^{k,\bar k}_{x,v}(\T^d
  \times \R^d)}^2.
\end{align*}
This proves that the norm remains finite by the Gronwall lemma and
concludes the proof.

\bibliographystyle{acm}
\bibliography{toy}
\end{document}